\documentclass[reqno,twoside]{amsart}
\usepackage{amsthm}
\usepackage{xcolor}
\usepackage[bookmarks=true]{hyperref}

\usepackage{amsfonts,amsmath,amssymb}
\usepackage{mathrsfs,mathtools}
\usepackage{enumerate}
\usepackage{esint}
\usepackage{graphicx}
\usepackage{bm}
\usepackage{commath}
\usepackage{esint}
\usepackage{caption}
\usepackage{float}
\DeclareMathAlphabet{\mathpzc}{OT1}{pzc}{m}{it}
\usepackage{color}

\newcommand\lfrac{(-\Delta)^s} 

\newcommand\RthLI[1][]{R_{\theta, \theta'}^{q_{#1}, \infty}}
\newcommand\RthLo[1][]{R_{\theta, \theta'}^{q_{#1},1}}  
 
\newcommand\ext{\Omega_e}
\newcommand\ns{\mathcal N_s} 

\newcommand{\rn}{\mathbb R^n}

\newcommand{\E}{\mathcal{E}}

\newcommand\sfracf[1]{\frac{#1}{|x-y|^{n+2s}}}

\newcommand\wtt{W_{\theta}}
\newcommand\wtp{W_{\theta'}}
\newcommand\utp{u_{\theta}^*}
\newcommand\rnstar{\mathbb R^n_*}

\newcommand\ddualstar{\mathcal{D}'_*}

\usepackage{caption}

\numberwithin{equation}{section}

\usepackage[dvips,bottom=1.4in,right=1in,top=1in, left=1in]{geometry}

\newtheorem{theorem}{Theorem}[section]
\newtheorem{proposition}[theorem]{Proposition}
\newtheorem{lemma}[theorem]{Lemma}

\newtheorem{assumption}[theorem]{Assumption}

\theoremstyle{definition}
\newtheorem{definition}[theorem]{Definition}

\newtheorem{remark}[theorem]{Remark}

\title[Unique continuation and reconstruction for the fractional Laplacian]{Reconstruction algorithms for the fractional Laplacian and applications to inverse problems}
\author{Ethan Rinaldo}
\address{Ethan Rinaldo, Université des Antilles, Laboratoire L.A.M.I.A., D\'epartement de Math\'ematiques et Informatique, Universit\'{e} des Antilles, Campus Fouillole, 97159 Pointe-\`a-Pitre, Guadeloupe, FWI}
\email{ethan.rinaldo@univ-antilles.fr}

    \author{Mahamadi Warma}
\address{Mahamadi Warma,  Department of Mathematical Sciences and The Center for Mathematics and Artificial Intelligence, George Mason University, Fairfax, VA 22030, USA}
\email{mwarma@gmu.edu}



\begin{document}
        \begin{abstract}
            We introduce two reconstruction schemes that enable the recovery of a function in the entire Euclidean space $\mathbb{R}^n$ from local data $(u|_W, [(-\Delta)^s u]|_W)$, where $W$ is an arbitrarily small nonempty open subset of $\rn$ and $(-\Delta)^s$ denotes the fractional Laplace operator of order $s\in (0,1)$. These procedures rely crucially on the weak Unique Continuation Property (UCP) for the fractional Laplacian.
We apply these schemes to two distinct inverse problems. Following the seminal work from Ghosh et al., the first one concerns the recovery of a potential (Calderón-type problem) from the fractional Schrödinger equation under nonlocal Robin-type exterior conditions. The second one involves recovering the solution of the space-fractional heat equation in $\mathbb{R}^n$ from localized time-dependent measurements within a ball.
To tackle these problems, we introduce new analytical tools such as a generalized weak Kelvin transform and a fractional Robin-to-Robin map. Finally, we provide numerical simulations for one of the reconstruction methods, illustrating the stability issues and the severe ill-posedness inherent to such inverse problems.
        \end{abstract}
        
			\maketitle
			\section{Introduction}
            In this paper we present two novel reconstruction algorithms that allow us to recover a function $u$ from the knowledge of the couple $(u|_U, [\lfrac u]|_U)$ where $U\subset\rn$ ($n\ge 1$) is an arbitrarily small nonempty open set and $\lfrac$ is the fractional Laplace operator defined as
            $$
                 \mathcal S(\rn) \rightarrow L^2(\rn) ,\qquad 
                 u  \mapsto \mathcal F^{-1}(|\xi|^{2s}\mathcal F(u)),
            $$
where $\mathcal S(\rn)$ denotes the Schwartz space, $\mathcal F$ is the Fourier transform and $\mathcal F^{-1}$ is the inverse Fourier transform. We shall give a more rigorous definition in Section \ref{Func-Set}. 

On the theoretical side, recent works from Ghosh, Salo \& Uhlmann \cite{GSU} have shown that the fractional Laplacian enjoys a unique continuation property (UCP), that is, if a function $u$ and its fractional Laplacian $\lfrac u$ both vanish on the same open subset of $\rn$, then $u$ is identically zero in the whole space $\rn$. This is a very distinct feature of the fractional Laplacian that does not apply to its local counterpart, the Laplace operator, and it illustrates the rigidity and nonlocality of the operator. Ghosh, R\"uland, Salo \& Uhlmann \cite{GRSU} used this property to prove uniqueness results in the context of the fractional Calder\'on problem. This seminal work laid the path for many others in the context of fractional inverse problems such as \cite{bhattacharyya2021inverse,cekic2020calderon,covi2020UCP,CalderonParabolic, railoZimmerman2023,RulandAnisotropic2023,SaloReview2017}, to cite only a few.

In the present work, we want to emphasize that an alternative and equivalent way to consider the UCP is that a function $u$ in $\rn$ is uniquely determined by the knowledge of the couple $(u, \lfrac u)$ in a bounded open subset of $\rn$. Thus, the question becomes. Is there a way to effectively reconstruct $u$ in the whole space $\rn$ from this couple?

This is precisely what we achieve here, proposing two distinct reconstruction procedures that allow us to recover a function from local fractional data. The first one, inspired by the original proof of the UCP for Riesz Potentials (\cite{riesz}), features the Kelvin transform, while the second one makes use of explicit formulas of Green functions for the fractional Poisson equation in the ball. To the best of our knowledge these results are new and may have fruitful applications in the context of fractional inverse problems. We give first some applications of the algorithms that we came across during the redaction of this article. In this research, we are interested in the resolution of a variant of the fractional Calder\'on problem where the exterior condition is a nonlocal Robin condition. Basically, we seek to identify the potential $q$ in the system
\begin{equation}\label{q_robin_frac}
				\begin{cases}
					\lfrac u+ qu  &= 0 \quad \text{in } \Omega \\	
					\nu_s (u) + \theta u &= f \quad \text{in } \ext:=\rn\backslash  \Omega,
				\end{cases}
\end{equation}
where $\nu_s$ would be a nonlocal counterpart of the normal derivative. Along the way, we come up with two different approaches to this problem. One where $\nu_s$ is represented by the nonlocal normal derivative $\ns$ (see e.g.  Dipierro, Ros-Oton \& Valdinoci \cite{dipierro}) given by
$$\ns u(x)=C_{n,s}\int_{\Omega}\frac{u(x)-u(y)}{|x-y|^{n+2s}}\;dy,\qquad x\in \ext,$$
and a second approach where $\nu_s$ is represented by the Dirichlet-to-Neumann map $\Lambda_q u := [\lfrac u]|_{\ext}$ (see e.g. \cite{GSU}). In both perspectives we propose a different version of a nonlocal Robin-to-Robin map as the measurable data, inspired by the operator introduced by Païvarinta \& Zubeldia \cite{paivarinta} in the local case of the Laplace operator. Both approaches also feature the UCP as their main tool  of resolution. The first case is treated in Section \ref{Inv-Pro} while the second one is treated in Section \ref{Sec-RAA} as an application of the reconstruction algorithms we propose.

Recently, there has been a growing interest in fractional operators due to their relevance in a wide spectrum of applications. There are several situations where a fractional equation gives a significantly better description than a classical local partial differential equation of the problem one wants to analyze. The fractional Laplace operator and its variants  usually describe anomalous diffusion. Examples in which fractional operators appear are models in turbulence (Bakunin \cite{bakunin2008}), population dynamics (De Roos \& Persson \cite{de2002}), image processing (Gilboa \& Osher \cite{gilboa2008}), laser design (Longhi \cite{longhi2015}), and porous media flow (Vazquez \cite{vazquez2012}). In addition, a number of stochastic models associated with fractional operators have been introduced in the literature to explain anomalous diffusion. Among them we quote the fractional Brownian motion, the continuous time random walk, the L\'evy flights, the Schneider gray Brownian motion, and more generally, random walk models based on evolution equations of single and distributed fractional order in  space (see e.g. Dubkov, Spagnolo \& Uchaikin \cite{dubkov2008}, Gorenflo, Mainardi \& Vivoli \cite{gorenflo2007}, Mandelbro \& Van Ness \cite{mandelbrot1968}, Schneider \cite{schneider1990}). In general, a fractional diffusion operator corresponds to a diverging jump length variance in the random walk. Finally, we can refer to Antil \& Rautenberg \cite{antil2019} and Weiss, van Bloemen \& Antil \cite{weiss2018} for the relevance of fractional operators in geophysics and imaging science.

            The rest of the paper is structured as follows. In Section \ref{Func-Set} we introduce the fractional order Sobolev spaces needed to study our problem and give rigorous definitions of the fractional Laplace operator and the associated nonlocal normal derivative. We also recall some integrations by parts formulas that are needed throughout the article. Other qualitative properties of the fractional Laplacian are also given. Our real work starts in Section \ref{Inv-Pro}. In Section \ref{two-sol} we introduce our two notions of solutions to the Robin problem \eqref{q_robin_frac} with $\nu_s = \ns$, prove their existence, uniqueness and regularity. In Section \ref{conv-DP} we prove two convergence results of solutions of the Robin problem to solutions of the Dirichlet problem. In Section \ref{Rob-to-Rob} we solve the inverse problem associated with \eqref{q_robin_frac}, that is, we recover the potential $q$ by using the Robin-to-Robin map for our two notions of weak solutions. In Section \ref{Kel-T} we introduce a weak version of the Kelvin transform, 
            and prove several properties of this weak version that will help us in the proof of the first reconstruction algorithm. Section \ref{Sec-RAA} is devoted to the reconstruction algorithms. We firstly use the introduced weak Kelvin transform and then secondly we make use of Green functions in the ball. Applications to the space-fractional heat equation and the fractional inverse Robin problem are given. Numerical illustrations that confirm our theoretical findings are also presented in Section \ref{Sec-Numerical}.

   \section{Functional settings}\label{Func-Set}
            We introduce the functional setting needed to study the inverse problems in Sections 3 and 5. We recall the definition of a few classical fractional order Sobolev spaces and introduce some new ones. We then proceed to define the fractional Laplace operator and give some associated results that will play an important role in solving the inverse problems. 
            Throughout the rest of the paper, without any mention, if $\Omega\subset\mathbb R^n$ ($n\ge 1$) is an open set, then we denote $\Omega_e:=\mathbb R^n\setminus\Omega$.
            Also, in the remaining part of this section, $\Omega\subset\mathbb R^n$ is an arbitrary nonempty open set and $0<s<1$ is a real number. In each result we shall specify if a regularity on the open set $\Omega$ is needed.

            \subsection{Fractional order function spaces} 
            We first recall the definition of Hölder continuous spaces. Let $j\geq 0$ be an integer and $\lambda \in (0, 1)$. Then, we let
            $$
            C^{j,\lambda}(\Omega) := \left \{ u\in C^j(\Omega) : \max_{0\leq |\alpha| \leq j} \sup_{x \in \Omega} |D^\alpha u(x)| + \max_{0\leq |\alpha| \leq j} \sup_{x \in \Omega} \frac{|D^\alpha u(x) - D^\alpha u(y)|}{|x-y|^\lambda} < \infty \right  \},
            $$
            and we endow it with the well-known natural norm. 
            
            We also recall the definition of the fractional order Sobolev space 
                $$
                H^{s}(\Omega) := \left \{u \in L^2(\Omega) : \int_\Omega \int_\Omega \sfracf{|u(x)-u(y)|^2}\,dy\,dx < \infty \right \}
                $$
                that we endow with the norm given by
                $$\|u\|_{H^{s}(\Omega)}:=\left(\int_{\Omega}|u|^2\;dx+\int_\Omega \int_\Omega \sfracf{|u(x)-u(y)|^2}\,dy\,dx\right)^{1/2}.$$
                It is well-known that $H^{s}(\Omega)$ is a Hilbert space and we have the continuous embedding 
                \begin{equation}\label{CE}
                H^{s}(\Omega)\hookrightarrow L^2(\Omega).
                \end{equation}
                Moreover, if $\Omega$ is assumed to be bounded and has a Lipschitz continuous boundary, then the continuous  embedding \eqref{CE} is also compact.

                We also define the spaces
                $H_0^s(\Omega):=\overline{\mathcal D(\Omega)}^{H^s(\Omega)},$
                where $\mathcal D(\Omega)$ denotes the space of all infinitely continuously differentiable functions with compact support in $\Omega$,
                and 
                $$\widetilde{H}_0^s(\Omega):=\left\{u\in H^s(\rn):\; u=0\mbox{ a.e. in } \rn\setminus\Omega\right\}=\{u\in H^s(\rn):\;\text{supp}(u)\subset \overline\Omega\}.$$
                 The dual spaces of $H_0^s(\Omega)$ and $\widetilde H_0^s(\Omega)$ are denoted by $H^{-s}(\Omega)$ and $\widetilde{H}^{-s}(\Omega)$, respectively.



\begin{definition}\label{ext-prop}
We say that $\Omega$ has the $H^s$-extension property, if there exists an extension operator $E_s:H^{s}(\Omega)\to H^{s}(\rn)$ such that $(E_su)|_{\Omega}=u$, for every $u\in H^{s}(\Omega)$.
\end{definition}

It is well-known (see e.g. Grisvard \cite[Chapter 2]{Gris}) that in this case, there is a constant $C=C(n,s,\Omega)>0$ such that 
\begin{equation}\label{ext-P}
    \|E_su\|_{H^{s}(\rn)}\le C\|u\|_{H^{s}(\Omega)}.
\end{equation}
It is also known that if $\Omega$ has a Lipschitz continuous boundary, then it has the $H^{s}$-extension property in the sense of Definition \ref{ext-prop} (see e.g. Di Nezza, Palatucci \& Valdinoci \cite[Theorem 5.4]{NPV} or \cite[Chapter 2]{Gris}).

Next, giving $g \in L^1(\ext)$, Dipierro, Ros-Oton \& Valdinoci \cite{dipierro} introduced the Sobolev type space 
$$ W^{s,2}_{g,\Omega} := \left \{ u: \rn \rightarrow \mathbb R : \E(u,u) + \int_{\Omega} u^2\, dx + \int_{\ext} |g|u^2\, dx < \infty \right \}
                    $$
                    where $\E(\cdot, \cdot)$ is the bilinear form given by
        \begin{equation}\label{form-E}
                \E(u,v) :=\frac{C_{n,s}}{2} \int\int_{(\rn\times \rn) \backslash \ext^2} \sfracf{(u(x) - u(y))(v(x)-v(y))} \, dy\;dx,
    \end{equation}
    and the normalization constant $C_{n,s}$ is given by
\begin{equation}\label{CN}
C_{n,s}:=\frac{s2^{2s}\Gamma\left(\frac{2s+n}{2}\right)}{\pi^{\frac
n2}\Gamma(1-s)},
\end{equation}
and $\Gamma$ is the usual Euler Gamma function (see e.g. Bjorland, Caffarelli \& Figalli \cite{BCF}, Caffarelli \& Silvestre \cite{Caf3}, Caffarelli, Roquejoffre \& Sire \cite{Caf1}, Caffarelli, Salsa \& Silvestre \cite{Caf2}, \cite{NPV}, and Warma \cite{War-DN1,War,warma2018approximate}).
                It has been shown in \cite[Proposition 3.1]{dipierro} that  $ W^{s,2}_{g,\Omega}$ endowed with the norm
                $$\|u\|_{ W^{s,2}_{g,\Omega}}:=\left( \E(u,u) + \int_{\Omega} u^2\, dx + \int_{\ext} |g|u^2\, dx\right)^{1/2}$$
                is a Hilbert space. This space has been used to study the fractional homogenous Neumann and Robin problems (see e.g. Claus \& Warma \cite{claus_warma} and \cite{dipierro}). We adapt it here for our own needs.  If $g=0$, then we only denote $W^{s,2}_{\Omega}:=W^{s,2}_{0,\Omega}$.
                We observe that functions in $W^{s,2}_{g,\Omega}$ do not necessary belong to $L^2(\rn)$. For that reason, we introduce the space
                    \begin{equation*}
                    V := \left \{ u \in L^2(\rn) : \int\int_{(\rn\times \rn) \backslash \ext^2}\sfracf{|u(x)-u(y)|^2}\;dy\;dx < \infty \right \},
                    \end{equation*}
                and we endow it with the norm given by
                \begin{equation}\label{norm-V}
                \|u\|_V = \left ( \|u\|_{L^2(\rn)}^2 + \E(u,u) \right )^{1/2}.
                \end{equation}
\begin{lemma}\label{lem-V}
	The space $V$ endowed with the inner product
            $$
           (u,v)_V :=  \E(u,v) + \int_{\rn} uv\, dx,
            $$
   with associated norm given in \eqref{norm-V} is a Hilbert space.
\end{lemma}

\begin{proof}
	We observe that 
    $V = L^2(\rn) \cap W^{s,2}_{\Omega}$.
    Thus, if $u\in V$, then, 
	\begin{equation}
		\|u\|_V = 0 \implies \|u\|_{L^2(\rn)} = 0\implies u = 0 \:  \text{ a.e. in } \rn.
	\end{equation}
	Let $(u_k)$ be a Cauchy sequence in $V$. Then, it is also a Cauchy sequence in $L^2(\rn)$ and in $W^{s,2}_{\Omega}$. In order to prove the completeness of $V$ one only needs to show that both limits are the same.  
	As it has been shown in \cite{dipierro}, as $k\to\infty$, the sequence $(u_k)$ converges a.e. and in the sense of the norm $\|\cdot \|_{W^{s,2}_{\Omega}}$ towards a function $v$ defined in $\rn$. Since $(u_k)$ is a Cauchy sequence in $L^2(\rn)$, it converges to some function $\tilde v \in L^2(\rn)$, as $k\to\infty$. Therefore, after a subsequence if necessary, we have that $(u_{k})$ converges a.e. to $\tilde v$, as $k\to\infty$. The uniqueness of the limit implies that $v = \tilde v \in L^2(\rn)$. 
\end{proof}

Next, let $\theta_0 \in \mathbb R^*_+:=(0,\infty)$ and $\rho \in L^1(\ext)$ with $\rho \geq 0$ a.e. in $\ext$. Let us set $\theta := \rho + \theta_0 \in L^1_{\rm loc}(\ext)$.  We introduce the space 
                \begin{equation}\label{esp=W}
                    \wtt  := \left \{ u:\rn \rightarrow \mathbb R\text{ measurable such that } \|u\|_{\wtt} < \infty \right \}, 
                    \end{equation}
            where
            \begin{equation}\label{norm-W}
            \|u\|_{\wtt} := \left(\E(u,u) + \|u\|_{L^2(\Omega)}^2 + \int_{\ext} |\theta|u^2 \;dx\right)^{\frac 12}.
            \end{equation}
            
            We have the following result.
            
            \begin{lemma}
             The space $W_\theta$ endowed with the inner product
            $$
           (u,v)_{\wtt} :=  \E(u,v) + \int_{\Omega} uv\, dx + \int_{\ext} |\theta|uv \, dx,
            $$
            with associated norm given in \eqref{norm-W} is a Hilbert space.
            \end{lemma}
            
            \begin{proof}
                We proceed as in the proof of Lemma \ref{lem-V}. Since $\theta = \rho + \theta_0$, it follows that the norm on $\wtt$ is equivalent to the norm
                $$
               \left( \E(u,u) + \int_{\rn} u^2 \, dx + \int_{\ext} \rho u^2\, dx\right)^{1/2}.
                $$
                Observe that the norm $ \theta_0\|u\|^2_{L^2(\ext)}+ \|u\|^2_{L^2(\Omega)}$ is equivalent to the norm $\|u\|_{L^2(\mathbb R^n)}^2$. This implies that
                $\wtt = V \cap W^{s,2}_{\rho,\Omega}$.
                As before, taking a Cauchy sequence $(u_k)\subset \wtt$, we get that $(u_k)$ is also a Cauchy sequence in $V$ and in $W^{s,2}_{\rho,\Omega}$. From both convergences, we can extract subsequences, if necessary, converging a.e. The uniqueness of the limit ensures that the limit is within the intersection of $V$ and $W^{s,2}_{\rho,\Omega}$. The completeness of $\wtt$ follows immediately.
            \end{proof}

To the best of our knowledge, it is the first time that the spaces $V$ and $W_\theta$ have been introduced.     
For more information about fractional order Sobolev spaces we refer the interested readers to \cite{NPV,Gris,dipierro,War}, and their references. Other variants of fractional order Sobolev spaces have been very recently introduced by Chill \& Warma \cite{Ch-Wa}.

            In order to simplify the redaction, we further denote the set of admissible parameters for the space $\wtt$ as follows: 
            \begin{equation}\label{Lambda}
            \Lambda_1 := \left \{ \theta \in L^1_{\rm loc}(\ext) : \exists \alpha > 0,\; \theta-\alpha \in L^1(\ext) \text{ and } \theta-\alpha>0  \text{ a.e. in } \ext \right \}.
            \end{equation}
        
\subsection{The Fractional Laplace operator and its properties}
For $0<s<1$ we set 
\begin{equation*}
\mathcal{L}_s^{1}(\mathbb R^n):=\left\{u:\mathbb R^n\rightarrow
\mathbb{R}\;\mbox{
measurable and }\;\int_{\mathbb R^n}\frac{|u(x)|}{(1+|x|)^{n+2s}}\;dx<\infty \right\} .
\end{equation*}%
For $u\in \mathcal{L}_s^{1}(\mathbb R^n)$ and $\varepsilon >0$, we let
\begin{equation*}
(-\Delta )_{\varepsilon }^{s}u(x):=C_{n,s}\int_{\{y\in \mathbb R^n: |y-x|>\varepsilon \}}
\frac{u(x)-u(y)}{|x-y|^{n+2s}}\;dy,\;\;x\in\mathbb R^n,
\end{equation*}%
where the normalization constant $C_{n,s}$ is given by \eqref{CN}.
 Then,  the fractional Laplacian 
$(-\Delta )^{s}$ is defined for $u\in \mathcal{L}_s^{1}(\mathbb R^n)$  by the formula
\begin{equation}
(-\Delta )^{s}u(x)=C_{n,s}\mbox{P.V.}\int_{\mathbb R^n}\frac{u(x)-u(y)}{|x-y|^{n+2s}}\;dy 
=\lim_{\varepsilon \downarrow 0}(-\Delta )_{\varepsilon
}^{s}u(x),\;\;x\in\mathbb R^n,\label{eq11}
\end{equation}%
provided that the limit exists for a.e. $x\in\mathbb R^n$ and P.V. stands for the Cauchy Principal Value.
We refer to  \cite{NPV} and the references therein for the class of functions on which the limit in \eqref{eq11} exists. It is important to mention that the fractional Laplacian has a proper inverse, called the Riesz potential and defined for $1<p<\frac{n}{2s}$ by 
\begin{equation*}
     I_{2s} : L^p(\rn)  \rightarrow  L^{\frac{np}{n-2sp}}(\rn),  \quad
       I_{2s} u(x)  :=  \int_{\rn}\frac{u(x)}{|x-y|^{n-2s}}\;dy. 
\end{equation*}

Next, following the notation of Silvestre \cite{silvestre_2006}, we define the space
            $$\overline{\mathcal S}_{2s}(\mathbb R^n):=\left\{f\in C^\infty(\rn):\;\;(1+|x|^{n+2s})f^{(k)}(x) \mbox{ is bounded for every } k\ge 0\right\}.$$

    We have the following result which proof can be found in \cite{silvestre_2006}.
            
            \begin{lemma}
                Let $f$ be in the well-known Schwartz space $\mathcal S(\mathbb R^n)$. Then, $\lfrac f \in \overline{\mathcal S}_{2s}(\mathbb R^n)$.
            \end{lemma}
            
            The fractional Laplacian is a nonlocal operator and, as such, it comes with a few peculiar properties. Among them, one that retains our attention is the UCP. It can be stated as follows (see e.g. \cite{GRSU,GSU}).
			
			\begin{lemma}\label{lemma_unique_continuation}
				Let $u \in H^{-r}(\rn)$ with $r\in \mathbb R$. If $\lfrac u$ and $u$ vanish in the sense of $H^{r-2s}(\rn)$  in an open subset of $\rn$, then $u \equiv 0$ in $\rn$.
			\end{lemma}
           
            The modern proof of this result includes advanced tools related to the fractional Laplacian, as, e.g., the Caffarelli-Silvestre extension \cite{Caf3}. But  similar results already exist, as Riesz's work on potential theory made use of an operator called Kelvin transform. We dive deeper in this matter in Sections 4 and 5. An alternative way to see this very rigid property is that any function $u \in H^r(\rn), r\in \mathbb R$, is the only function that corresponds to the data $(u|_U, [\lfrac u] |_U)$, where $U\subset \rn$ is an arbitrary open set. This leads to the two following reconstruction results, an injectivity result and a reconstruction theorem that go hand in hand, and have been proved in \cite{GRSU}.

            \begin{lemma} \cite[Lemma 2.2]{GRSU}
                Let $\Omega \subset \rn$ be a bounded open set and $W \subset \rn$ a nonempty open set such that $\overline W \cap \overline \Omega = \emptyset$. Consider the operator 
                $$
                L: \widetilde H_0^s(\Omega) \rightarrow H^{-s}(W), \qquad u\mapsto [\lfrac u]|_{W}. 
                $$
                Then, L is a compact and an injective operator with dense range.
            \end{lemma}

    \begin{theorem}\cite[Theorem 2]{GRSU}\label{thm_grsu_tikonov}
				Let $\Omega \subset \rn$ be a bounded open set and  $W$  a nonempty open set such that $\overline{W} \subset \ext$. Any function $v \in H^{s}(\rn)$ 
                with $\mathrm{supp}(v) \subset \overline \Omega$ is uniquely determined by the knowledge of $[\lfrac v]|_{W}$. The function $v$ can be reconstructed thanks to the formula
				$$
				v = \lim_{\alpha \rightarrow 0} v_\alpha \quad (\text{in } H^{s}(\rn)),
				$$
				where $v_\alpha$ ($\alpha>0$), is the unique solution of the minimization problem
				\begin{equation}\label{PT}
				v_\alpha = \mathrm{arg } \min_{w\in \widetilde H_0^s(\Omega)} \left [ \|[\lfrac w]|_{W} - [\lfrac v]|_{W}\|_{\widetilde H^{-s}(W)}^2 + \alpha \|w\|^2_{H^s(\rn)} \right ].
				\end{equation}
			\end{theorem}

     In order to deal with the nonlocal fractional Robin condition, we next introduce a nonlocal analogue of the normal derivative (see e.g. \cite{dipierro}).  For a function $u\in W^{s,2}_{\Omega}$, the nonlocal operator $\mathcal N_s$ is defined as 
    \begin{equation}\label{op-Ns}
    \mathcal N_su(x)=\frac{C_{n,s}}{2}\int_\Omega\frac{u(x)-u(y)}{|x-y|^{n+2s}}\;dy,\qquad x\in\ext.
    \end{equation}

    We have the following integration by parts formula taken from  \cite[Proposition 2.6] {claus_warma}.
    
            \begin{proposition}\label{prop_parts_integration}
            Let $\Omega\subset\rn$ be a bounded open set.
	Let $u \in W^{s,2}_{\Omega}$ be such that $\lfrac u \in L^2(\Omega)$ and $\ns u \in L^2(\ext)$. Then, for all $v \in V$ we have that
	\begin{equation}\label{Int-Part}
		\int_\Omega v \lfrac u \,dx = \E(u,v) - \int_{\ext}v\ns u\,dx,
	\end{equation}
    where we recall that $\mathcal E$ is given in \eqref{form-E}.
\end{proposition}

 The following result will be useful in the coming sections.

\begin{proposition}\label{prop_parts_formula_distrib}
	Let $u \in V$. Then, the  identity 
	\begin{equation*}
		\langle\lfrac u, \varphi\rangle_{\mathcal D'(\Omega), \mathcal D(\Omega)} = \E(u, \varphi), \, \forall \varphi \in \mathcal D(\Omega),
	\end{equation*}
    holds in the sense of distributions.
    
\end{proposition}

\begin{proof}
	The duality pairing $\langle \cdot, \cdot \rangle_{\mathcal D'(\Omega), \mathcal D(\Omega)}$ in the distributional sense $\mathcal D'(\Omega)$ is given by
	\begin{equation*}
		\langle\lfrac u, \varphi\rangle_{\mathcal D'(\Omega), \mathcal D(\Omega)} := \int_{\rn} u \lfrac \varphi\,dx,\qquad \varphi\in\mathcal D(\Omega).
	\end{equation*}
	Splitting the domain of the integral between $\Omega$ and $\ext$ and applying Proposition \ref{prop_parts_integration}, we obtain that
	\begin{align*}
		\langle\lfrac u, \varphi\rangle_{\mathcal D'(\Omega), \mathcal D(\Omega)} = &\E(u, \varphi) - \int_{\ext} u \ns \varphi\,dx + \int_{\ext} u \lfrac \varphi\,dx\notag\\
		 =&\E(u, \varphi) - C_{n,s}\int_{\ext}u(x) \int_{\Omega}\sfracf{-\varphi(y)}\,dy\;dx \notag\\
        &+C_{n,s} \int_{\ext}u(x) \text{ P.V.} \int_{\Omega}\sfracf{\varphi(x)-\varphi(y)}\,dy\;dx.
	\end{align*}
	Observe that for the last integral, the Cauchy principal value is unnecessary as $x \in \ext$ and $y\in \Omega$. Thus,
	\begin{align*}
		\langle\lfrac u, \varphi\rangle_{\mathcal D'(\Omega), \mathcal D(\Omega)} =& \E(u, \varphi) + C_{n,s}\int_{\ext}u(x) \int_{\Omega}\frac{\varphi(y)}{|x-y|^{n+2s}}\,dy\;dx \notag\\ 
        &- C_{n,s}\int_{\ext}u(x) \int_{\Omega}\frac{\varphi(y)}{|x-y|^{n+2s}}\,dy\;dx\notag\\
		 =& \E(u,\varphi),
	\end{align*}
 for all $\varphi\in \mathcal D(\Omega)$.
\end{proof}

\section{Direct and inverse fractional Robin problems}\label{Inv-Pro}

    Here, we are interested in solving both the direct and inverse problems associated to the fractional Schr\"odinger equation with the nonlocal Robin exterior conditions. First, we discuss the existence, uniqueness and regularity of solutions to the Robin problem
    \begin{equation}\label{eq_robin_frac}
            \begin{cases}
                \lfrac u+ qu  = 0 \quad &\text{ in } \Omega ,\\	
                \ns u + \theta u = \phi \quad & \text{ in } \ext,
            \end{cases}
        \end{equation}
     where here, $\theta$ is a function and not a constant. For that purpose we introduce two notions of weak solutions to our problem which nature depends on the space where the function $\theta$ belongs. 
     We then proceed to tackle the inverse problem of determining the potential $q$ from exterior measurements. Our results follow closely those from  \cite{GRSU} for the Dirichlet exterior condition, but we build a new nonlocal Robin-to-Robin map inspired by the Robin-to-Robin map featured in \cite{paivarinta} for the local case of the Laplace operator. 

     It is important to mention that we are proposing an alternative way to deal with the inverse Robin problem in Section 5, with a Dirichlet-to-Robin map instead of a Robin-to-Robin map.

    \subsection{ $L^{\infty}$ and $L^{1}$-weak solutions}\label{two-sol}
    We make the following assumption on the potential $q$.

    \begin{assumption}\label{Ass-q}
    The function $q\in L^\infty(\Omega)$ and there is a constant $q_0>0$ such that $q\ge q_0>0$ a.e. in $\Omega$.
    \end{assumption}
    
     Next, we introduce our two notions of weak solutions.
    
    \begin{definition}[\bf $L^\infty$-weak solution]\label{def_linf_sol}
         Let $\theta_0 \in \mathbb R^{*}_+$ and $\theta \in L^\infty(\ext)$ with $ \theta \ge  \theta_0$ a.e. in $\ext$ and let $\phi \in L^2(\ext)$. We say that $u \in V$ is an $L^\infty$-weak solution to the Robin problem \eqref{eq_robin_frac} if the equality
            $$
            \E(u,v) + (qu,v)_{L^2(\Omega)} + (\theta u,v)_{L^2(\ext)} =  (\phi, v)_{L^2(\ext)}, 
            $$
            holds for every $v\in V$.
        \end{definition}

        \begin{definition}[\bf $L^1$-weak solution]\label{def_l1_sol}
            Let $\rho \in L^1(\ext)$ be such that $\rho \geq 0$ a.e. in $\ext$, $\theta_0 \in \mathbb R^*_+$ and $\theta := \rho + \theta_0$. Let $\phi \in L^2(\ext)$. We say that $u \in \wtt$ is an $L^1$-weak solution to the Robin problem \eqref{eq_robin_frac} if the equality
            $$
            \E(u,v) + (qu,v)_{L^2(\Omega)} + (\theta u,v)_{L^2(\ext)} =  (\phi, v)_{L^2(\ext)},
            $$
            holds for every $v \in \wtt$.
        \end{definition}

        Observe that in Definition \ref{def_linf_sol}, $\theta\in L^\infty(\ext)$, while in Definition \ref{def_l1_sol},  $\theta\in L^1_{\rm loc}(\ext)$.
        
 The following theorem is the first main result of this section.
 
        \begin{theorem}\label{thm_linf_sol}
        Let $\theta_0 \in \mathbb R^{*}_+$ and $\theta \in L^\infty(\ext)$ with $\theta \ge  \theta_0$ a.e. in $\ext$. Let $\phi \in L^2(\ext)$ and let $q$ satisfy Assumption \ref{Ass-q}. Then, the Robin problem \eqref{eq_robin_frac} has a unique $L^\infty$-weak solution $u\in V$. Moreover, 
        \begin{equation}\label{eq-a.e.}
        \mathcal N_su+\theta u=\phi \;\mbox{ a.e. in } \ext,
        \end{equation}
        and $u$ is also a solution in the sense of the distributions, that is,
        \begin{equation}\label{SEQ}
        \langle (-\Delta)^su,\varphi\rangle_{\mathcal D'(\Omega),\mathcal D(\Omega)} +(qu,\varphi)_{L^2(\Omega)}=0,\qquad \forall\;\varphi\in \mathcal D(\Omega).
        \end{equation}
        \end{theorem}
        
        \begin{proof}
            We apply the classical Lax-Milgram Theorem. It is easy to see that there is a constant $C>0$ such that for every $u,v\in V$, 
            $$\Big|\E(u,v) + (qu,v)_{L^2(\Omega)} + (\theta u,v)_{L^2(\ext)}\Big|\le C\|u\|_V\|v\|_V.$$
           Since  $\theta$ is bounded from below we also get the coercivity. That is, there is a constant $C=C(n,s,q,\theta)>0$ such that for every $u\in V$, 
            $$\E(u,u) + \int_\Omega qu^2\, dx +  \int_{\ext} \theta u^2\;dx \geq \E(u,u) + \min(\theta_0, q_0)\|u\|_{L^2(\rn)}^2 \geq C\|u\|_V^2.$$
            It is easy to see that the functional $F:V\to\mathbb R$ given by $ F(v):=\int_{\ext}\phi v\;dx,$
            is linear and continuous. Hence, we have the existence and uniqueness of an $L^\infty$-weak solution $u$.  
           
           Next, we show \eqref{eq-a.e.}.
            We notice that
            	\begin{align*}
		\E(u,\varphi) = &\frac{C_{n,s}}{2}\int_{\Omega}\int_{\Omega} \sfracf{(u(x)-u(y))(\varphi(x)-\varphi(y))}\,dy\;dx \,
        \\ &+C_{n,s}\int_{\Omega}\int_{\ext} \sfracf{(u(x)-u(y))(\varphi(x)-\varphi(y))}\,dy\;dx.
	\end{align*}
            Thus, taking $\varphi \in \mathcal D(\ext)$ we obtain that
            $$
            \E(u,\varphi) = C_{n,s}\int_{\Omega}\int_{\ext} \sfracf{(u(x)-u(y))(-\varphi(y))}\,dy\;dx.
            $$
            Interchanging the order of the integrals we obtain that
            \begin{equation}\label{int-E}
            \E(u,\varphi) = C_{n,s}\int_{\ext} \varphi(x)\int_{\Omega} \sfracf{u(x)-u(y)}\,dy\;dx=\E(u,\varphi) =  \int_{\ext}\varphi(x) \ns u(x)\,dx.
            \end{equation}
            We can conclude that, if $u$ is an $L^\infty$-weak solution of  \eqref{eq_robin_frac}, then \eqref{eq-a.e.} holds.
            One may notice that this also proves that $\ns u \in L^2(\ext)$, since $\theta u$ and $\phi$ belong to $L^2(\ext)$. 
            
            Finally, regarding the Schr\"odinger equation inside $\Omega$, we take a test function $\varphi \in \mathcal D(\Omega)$. Applying Proposition \ref{prop_parts_formula_distrib} we get that $\langle\lfrac u, \varphi\rangle_{\mathcal D'(\Omega), \mathcal D(\Omega)} = \E(u,\varphi)$, and
            we can deduce that \eqref{SEQ} holds.
        \end{proof}

        Our second main result of this section is the following.

        \begin{theorem}\label{thm_l1_sol}
        Let $\rho \in L^1(\ext)$ be such that $\rho \geq 0$ a.e. in $\ext$, $\theta_0 \in \mathbb R^*_+$, $\theta := \rho + \theta_0$, $\phi \in L^2(\ext)$,  and let $q$ satisfy Assumption \ref{Ass-q}.
        Then, the Robin problem \eqref{eq_robin_frac} has a unique $L^1$-weak solution $u\in \wtt$. Moreover, $\ns u+\theta u=\phi$ a.e. in $\ext$, and  the $L^1$-weak solution is also a solution in the sense of distributions.
        \end{theorem}
        
        \begin{proof}
            As before, we use the Lax-Milgram Theorem. There is a constant $C=C(q,\theta)>0$ such that for every $u,v\in\wtt$, 
            $$
            \begin{aligned}
                \Big |\E(u,v) + (qu,v)_{L^2(\Omega)} + (\theta u,v)_{L^2(\ext)} \Big | \leq& \|u\|_{V}\|v\|_{V} + \|q\|_{L^\infty(\Omega)}\|u\|_{L^2(\Omega)} \|v\|_{L^2(\Omega)} \\
                &+  \int_{\ext}\left|\theta_0uv\right|\;dx + \int_{\ext}\left|\rho uv\right| \, dx \\
                \leq&  C\|u\|_{\wtt} \|v\|_{\wtt}.
            \end{aligned}
            $$
            Regarding the coercivity we have that for every $u\in\wtt$,
            $$
            \begin{aligned}
                \E(u,u) + (qu,u)_{L^2(\Omega)} + (\theta u,u)_{L^2(\ext)} &\geq \E(u,u) + q_0\|u\|_{L^2(\Omega)}^2 + \int_{\ext} \theta u^2 \, dx 
                 \geq \min(1, q_0, \theta_0) \|u\|_{\wtt}^2.
            \end{aligned}
            $$
            Also the functional $F: W_\theta\to\mathbb R$ given by
            $F(v)=\int_{\ext}\phi v\;dx$
            is linear and continuous.
            We have shown the existence and the uniqueness of an $L^1$-weak solution $u$. The rest of the proof follows as the proof of  Theorem \ref{thm_linf_sol}.  
        \end{proof}

        \begin{remark}
            We observe that, if $\theta \in L^\infty(\ext)$, then every $L^1$-weak solution is also an $L^\infty$-weak solution. This follows from the fact that, in that case, $\wtt \subset V$.
        \end{remark}
        
        \subsection{Convergence to the Dirichlet problem}\label{conv-DP}
            Considering only the weak formulations, the Dirichlet problem, as stated for example in \cite{GRSU}, and the Robin problem, seem unrelated. Especially as Robin solutions belong to $V$ or subspaces of $V$, and Dirichlet solutions belong to $H^s(\rn)$. However, both solve the Schr\"odinger equation in the sense of distributions in $\Omega$. We show that these problems are related in a stronger sense.

            First, we recall the notion of solutions to the Dirichlet fractional Schr\"odinger equation.

             \begin{definition}
             Let $\Omega\subset\rn$ be a bounded domain with a Lipschitz continuous boundary.
             Given $\phi\in H^s(\ext)$, we consider the Dirichlet problem
               \begin{equation}\label{Dir-P}
                \begin{cases}
					\lfrac u+ qu  &= 0 \quad \text{ in  } \Omega \\	
					u &= \phi \quad \text{ in  } \ext.
				\end{cases}
                \end{equation}
                A function $u_d \in H^s(\rn)$ is said to be a weak solution 
                of  \eqref{Dir-P} if
                \begin{equation*}
                \int_{\rn}(-\Delta)^{s/2} u_d(-\Delta)^{s/2} w\;dx + \int_{\Omega}qu_dw\;dx = 0, \quad \forall w \in  \widetilde H_0^s(\Omega),
                \end{equation*}
                and $u_d-\Phi \in \widetilde H_0^s(\Omega)$, where $\Phi\in H^s(\rn)$ is an extension of $\phi$ in the sense of Definition \ref{ext-prop}.
                \end{definition}

                \begin{remark}
                We observe the following.
                \begin{enumerate}
                \item[(a)] Since we have assumed that $\Omega$ is a domain and has a Lipschitz continuous boundary, then it is well-known (see e.g. \cite[Chapter 2]{Gris}) that $\ext$ has the $H^s$-extension property.
                \item[(b)] If $\phi \in \widetilde H_0^s(\ext)$, then no regularity assumption on $\Omega$ is needed.
                \item[(c)] From the condition $u_d-\Phi \in \widetilde H_0^s(\Omega)$, letting $v:=u_d-\Phi$ we have that $ v\in \widetilde H_0^s(\Omega)$  and
                \begin{equation*}
                  \int_{\rn}(-\Delta)^{s/2} v(-\Delta)^{s/2}w\;dx + \int_{\Omega}qvw\;dx = -\int_{\rn}(-\Delta)^{s/2} \Phi (-\Delta)^{s/2} w\;dx -\int_{\Omega}q\Phi w\;dx,
                \end{equation*}
                 for every $w \in \widetilde H_0^s(\Omega)$.
                Letting the functional $F:\widetilde H_0^s(\Omega)\to\mathbb R$ be given by
                $$F(w):=-\int_{\rn}(-\Delta)^{s/2} \Phi (-\Delta)^{s/2} w\;dx -\int_{\Omega}q\Phi w\;dx,$$
                we have that $F\in  \widetilde 
                H^{-s}(\Omega)$ and the function $v$ is naturally a weak solution of the Dirichlet problem with source term
                \begin{equation}\label{DPST}
                \begin{cases}
						\lfrac v+ qv  = F \quad &\text{ in } \Omega \\	
						v = 0 \quad &\text{ in } \ext.
					\end{cases}
                \end{equation}
               \item[(d)]  We do know that functions in $V$ are such that their restrictions to $\Omega$ belong to $H^s(\Omega)$. Thus, we aim to find out whether or not it is possible to make solutions of the Robin problem converge towards solutions of the homogeneous exterior Dirichlet problem in a somewhat natural way. In order to do so, we need to point out an existing relationship between the weak formulations of both problems.                

                \item[(e)] The bilinear form
                $$
                \int_{\rn}(-\Delta)^{s/2} u_d (-\Delta)^{s/2} w\;dx
                $$
                is actually equal to the bilinear form
                $$
                   \frac{C_{n,s}}{2} \int_{\rn}\int_{\rn} \sfracf{(u_d(x)-u_d(y))(w(x)-w(y))}\,dx\; dy.
                $$
                But, since $w \in \widetilde H_0^s(\Omega)$, we know that $\mathrm{supp }(w)\subset \overline \Omega$, and thus the domain of integration can be rewritten as $\mathbb 
                R^{2n} \backslash \ext^2$. This means that for $w \in \widetilde H_0^s(\Omega)$, the bilinear form coincides with $\E(u_d, w)$, where $\E$ is given in \eqref{form-E}.
                Therefore, $v$ solution of \eqref{DPST} satisfies the equality
                $$
                \E(v,w) + (qv,w)_{L^2(\Omega)} = \langle F,w\rangle_{\widetilde H^{-s}(\Omega), \widetilde H_0^s(\Omega)},\qquad \forall w \in \widetilde H_0^s(\Omega).
                $$  
                \end{enumerate}
                \end{remark}

                Before studying the convergence, we recall this well-known topological result. 

                \begin{theorem}\label{theo-FA}
                    Let $(x_k)_{k\in\mathbb N}$ be a sequence in a topological space $X$ such that every subsequence $(x_{k_j})$ of $(x_k)$  converges to $x \in X$, as $j\to\infty$. Then, the full sequence $(x_k)$ converges to x, as $k\to\infty$. 
                \end{theorem}
                
                \begin{proof}
                    Suppose by contradiction that $(x_k)$ does not converge to $x$. Then, there exists an open neighborhood $U$ of $x$ such that $\forall j > 0,\exists k_j > j$ such that $x_{k_j} \notin U$. Thus, taking $(x_{k_j})$ as a subsequence, we can see that no subsequence of $(x_{k_j})$ converges to $x$ since $(x_{k_j}) \cap U = \emptyset$, which contradicts the assumption.
                \end{proof}

Now, we are ready to state and prove the convergence results.

\begin{theorem}\label{conv-R-to-D}
Let $\theta \in \mathbb R^*_+$ and $F \in V^\star$, the dual space of $V$. 
Let $q$ satisfy Assumption \ref{Ass-q}.
Let $f\in L^2(\ext)$ and
                let $u_\theta\in V$ be the $L^\infty$-weak solution of the  nonhomogeneous exterior Robin problem with right-hand side,
                \begin{equation}\label{NHRP}
                \begin{cases}
					\lfrac u_\theta+ qu_\theta  &=  F \quad \text{in } \Omega, \\	
					\ns u_\theta + \theta u_\theta &= f \quad \text{in } \ext.
			\end{cases}
                \end{equation}
                 Then, $u_\theta$ converges weakly to a function $\tilde u\in \widetilde H_0^s(\Omega)$ and strongly in $L^2(\rn)$, as $\theta \to\infty$, and $\tilde u$ is the weak solution of the homogeneous exterior Dirichlet problem
                \begin{equation}\label{HDP}
                \begin{cases}
						\lfrac \tilde u+ q\tilde u  =  F \quad &\text{ in  } \Omega \\	
						\tilde u = 0 \quad  &\text{ in } \ext.
					\end{cases}
                \end{equation}
                \end{theorem}

                \begin{proof}
                We denote the duality map between $V^\star$ and $V$ by $\langle\cdot,\cdot\rangle_{V^\star,V}$. 
                 We recall that from the definition of weak solutions of \eqref{NHRP}, we have the equality
                \begin{equation*}
                    \E(u_\theta, u_\theta) + \int_\Omega qu_\theta^2\; dx + \theta \|u_\theta\|_{L^2(\ext)}^2  = \langle F, u_\theta\rangle_{V^\star,V} +(f,u_\theta)_{L^2(\ext)}.
                \end{equation*}
                From the above equality we get the estimate
                   \begin{equation*}
                     \E(u_\theta, u_\theta)+
                    q_0 \int_{\Omega}u_\theta^2\,dx + \theta \|u_\theta\|^2_{L^2(\ext)} \leq \| F\|_{V^\star}\|u_\theta\|_{V} + \|f\|_{L^2(\ext)} \|u_\theta\|_{L^2(\ext)}.
                \end{equation*}
                Since $\theta \rightarrow \infty$, we may suppose that $\theta > q_0$ and $\theta > 1$.
                Thus, we have the  estimate
                $$
                \|u_\theta\|_V \leq \frac{1}{\min(1, q_0)}\left (\|F\|_{V^\star} + \|f\|_{L^2(\ext)} \right ).
                $$
                Since $u_\theta$ is bounded in norm in the Hilbert space $V$, we can extract a subsequence denoted by $\utp$ that converges weakly towards some $\tilde u \in V$, as $\theta\to\infty$. Since the inclusion $V\subset L^2(\Omega)$ is compact, it follows that the convergence is strong in $L^2(\Omega)$.
                We need to prove that $\tilde u$ is a weak solution of \eqref{HDP}.
    
                First of all, we mention that
                $$
                \|u_\theta\|_{L^2(\ext)}^2 \leq \frac{1}{\theta} \left ( \|F\|_{V^\star} \|u_\theta\|_{V} +\|f\|_{L^2(\ext)}\|u_\theta\|_{L^2(\ext)} \right ).
                $$
                Since $\|u_\theta\|_{L^2(\rn)}$  and $\|u_\theta\|_V$ are bounded independently of $\theta$, we can conclude that $u_\theta$ converges strongly to $0$ in $L^2(\ext)$, as $\theta\to\infty$.   Since $V$ is continuously embedded in $L^2(\rn)$, it follows that  $\tilde u = 0\,\text{ a.e}$ in $\ext$. In other words, we have that $\tilde u \in \widetilde H_0^s(\Omega)$. We also have that $\utp$ converges strongly to $\tilde u$ in $L^2(\rn)$, as $\theta\to\infty$.

                Now, taking $w \in \widetilde H_0^s(\Omega)$ and using the  weak convergence in $V$, we can deduce that, as $\theta\to\infty$,
                $$
                 \E(\utp, w) + (q\utp, w)_{L^2(\Omega)} \longrightarrow \E(\tilde u, w) + (q\tilde u, w)_{L^2(\Omega)}. 
                $$
                Since 
                $$\E(\utp, w) + (q\utp, w)_{L^2(\Omega)} = \langle F, w\rangle_{V^\star,V}, \qquad \forall\ w \in \widetilde H_0^s(\Omega),$$ 
                we have that
                $$
                \E(\tilde u, w) + (q\tilde u,w)_{L^2(\Omega)} =  \langle F, w\rangle_{V^\star,V}, \qquad \forall\; w \in \widetilde H_0^s(\Omega).
                $$
                Hence, $\tilde u$ is the unique weak solution of the  Dirichlet problem \eqref{HDP}.
                As there was no specifications on $\utp$, we can conclude that every converging subsequence of $u_\theta$ admits $\tilde u$ as a limit. It follows from Theorem \ref{theo-FA} that the whole sequence $u_\theta$ converges to $\tilde u$ in $\widetilde H_0^s(\Omega)$, as $\theta\to\infty$.
                \end{proof}
                
                 The nonhomogeneous exterior case is in fact more tricky, and we can only prove a strong convergence in $L^2(\rn)$ as we show next.
                 
                 \begin{theorem}\label{conv-rn}
                 Let $\Omega\subset\rn$ be a bounded domain with a Lipschitz continuous boundary.
               Let $q$ satisfy Assumption \ref{Ass-q}. Let $z\in H^s(\ext)$ and let $u\in H^s(\mathbb R^n)$ be the unique weak solution of the  Dirichlet problem
         \begin{equation}\label{DP2}
         \begin{cases}
                 \lfrac u+ qu  = 0\quad &\text{ in  } \Omega \\	
                 u = z \quad & \text{ in } \ext.
     \end{cases}
         \end{equation}
         Let $\theta\in\mathbb R_+^\star$ and $u_\theta\in V$ be the unique weak solution of the Robin problem
         $$
         \begin{cases}
              \lfrac u_\theta+ q u_\theta  =  0 \quad &\text{ in } \Omega, \\
              \ns u_\theta +  \theta u_\theta  = \theta z \quad &\text{ in } \ext.
         \end{cases}$$
  If $\ns u\in L^2(\ext)$, then $u_\theta$ converges strongly  to $u$ in $L^2(\rn)$, as $\theta\to\infty$.
  \end{theorem}

  \begin{proof}
  We use a similar argument as the one given by Antil, Khatri \& Warma \cite[Theorem 6.3]{AKW-Inv}. 
          Using the integration by parts formula in Proposition \ref{prop_parts_integration},  we have the following equalities for every $v\in V$:
                \begin{align}\label{ed-ali}
        \E( u-u_\theta, v) &+ (q( u-u_\theta), v)_{L^2(\Omega)} + \theta ( u-u_\theta, v)_{L^2(\ext)}\notag \\
         = &(\lfrac ( u - u_\theta), v)_{L^2(\Omega)} + (\ns ( u - u_\theta), v)_{L^2(\ext)} \notag\\
         &+ (q( u-u_\theta), v)_{L^2(\Omega)} + \theta( u-u_\theta, v)_{L^2(\ext)}\notag \\
         = & ((\lfrac +q) u, v)_{L^2(\Omega)} - ((\lfrac +q)u_\theta, v)_{L^2(\Omega)} \notag\\
        & + (\ns ( u - u_\theta), v)_{L^2(\ext)} + \theta( u-u_\theta, v)_{L^2(\ext)} \notag\\
         =& (\ns  u, v)_{L^2(\ext)} + \theta( u, v)_{L^2(\ext)} - (\ns u_\theta + \theta u_\theta, v)_{L^2(\ext)} \notag\\
         = &(\ns  u, v)_{L^2(\ext)} + (\theta z - \theta z, v)_{L^2(\ext)} \notag \\
         =& (\ns  u, v)_{L^2(\ext)}. 
         \end{align}
         Thus, taking $v := u - u_\theta$ as a test function in \eqref{ed-ali} and using the assumption that $\ns u\in L^2(\ext)$, we get the estimates
         $$
         \theta \|\ u - u_\theta \|^2_{L^2(\ext)} \leq \int_{\ext} \left|( u - u_\theta) \ns  u\right|\,dx \leq \| u - u_\theta \|_{L^2(\ext)} \|\ns  u\|_{L^2(\ext)}.
         $$
         Dividing both sides by $\theta$, we get the strong convergence of $u_\theta$ towards $u$ in $L^2(\ext)$, as $\theta\to \infty$. 
         
         Similarly, from \eqref{ed-ali} we also get that
         $$
         q_0 \|\ u - u_\theta \|^2_{L^2(\Omega)} \leq \int_{\ext} \left|( u - u_\theta) \ns  u\right|\,dx \leq \| u - u_\theta \|_{L^2(\ext)} \|\ns  u\|_{L^2(\ext)}.
         $$
         This implies the strong convergence of $u_\theta$ in $L^2(\Omega)$, as $\theta\to \infty$.
         \end{proof}

We conclude this subsection with the following remark.

\begin{remark}
We observe the following facts. 
\begin{enumerate}
    \item[(a)] The result in Theorem \ref{conv-rn} has been proved in \cite[Theorem 6.3]{AKW-Inv} without making sure that any weak solution $u$ of the Dirichlet problem \eqref{DP2} satisfies $\ns u\in L^2(\ext)$.

    \item[(b)] We do not have a proof that for any weak solution $u$ of \eqref{DP2}, we have that $\ns u\in L^2(\ext)$. We can conjecture that this can be true if $\Omega$ and the given function $z$ are smooth as in Theorem \ref{conv-rn}. This is always the case in the local case $s=1$. This deserves a clarification in the nonlocal case $0<s<1$.
\end{enumerate}
\end{remark}
        
            $$
            $$
			\subsection{The inverse problem using the Robin-to-Robin map}\label{Rob-to-Rob}
            The main concern of this section is to solve the inverse problem \eqref{q_robin_frac}, that is, to recover the potential $q$. Our strategy here is to use the Robin-to-Robin map for both $L^\infty$- and $L^1$-weak solutions.
            
\subsubsection{The inverse problem for $L^{\infty}$-weak solutions}
		We proceed to solve the inverse problem within the $L^{\infty}$-weak solutions paradigm. 
        
        \begin{definition}[\bf $L^\infty$-Robin-to-Robin map]\label{def_linf_rmap}
       Let  $\theta, \theta' \in L^\infty(\ext)$ be such that $\theta-\theta' \neq 0$ a.e. in $\ext$ and let $f \in L^2(\ext)$. Let $u$ be the $L^\infty$-weak solution to the Robin problem \eqref{q_robin_frac} with given data $(f, \theta)$ and the operator $\nu_s=\mathcal N_s$. We call the $L^\infty$-Robin-to-Robin map, the mapping defined by
			$$
			\RthLI : L^2(\ext) \rightarrow L^2(\ext),\qquad \ns u + \theta u \mapsto \ns u + \theta' u.
			$$
		This can be rewritten in a more simple way as 
			\begin{equation*}
				\RthLI : L^2(\ext) \rightarrow L^2(\ext),\qquad f\mapsto  f + (\theta'-\theta)u.
			\end{equation*}
        \end{definition}

The following theorem is the main result of this subsection.
			
			\begin{theorem}
				Let $\Omega \subset \rn$ be a nonempty bounded open set with a Lipschitz continuous boundary. Let $\theta_0 \in \mathbb R^{*}_+$ and $\theta \in L^\infty(\ext)$ with $ \theta \ge  \theta_0$ a.e. in $\ext$.  Let also $q\in C(\overline\Omega)$ be such that $q\ge q_0>0$ for some constant $q_0$, and let $u\in V$ be the unique $L^\infty$-weak solution to the Robin problem \eqref{q_robin_frac}. Suppose that 
                \begin{itemize}
                \item $0<s<1/2$, or; 
                \item $1/2\le s<1$, the Robin-to-Robin data $(f, \RthLI f)$ are of $H^{s}$ regularity in a neighborhood of $\Omega$, and that $\frac{1}{\theta'-\theta}$ is bounded and Lipschitz continuous near $\partial \Omega$. 
                \end{itemize}
                Then, $u|_{\Omega}$ can be reconstructed with the knowledge of $(f, \RthLI f) \in (L^2(\ext))^2$, and
                \begin{equation}\label{eq-q}
                q(x) = - \frac{\lfrac u(x)}{u(x)}, \;\forall\;x  \in \Omega.
                \end{equation}
			\end{theorem}
            
			\begin{proof}
				Knowing $f \in L^2(\ext)$ and $\RthLI f$ we also know $(\theta'-\theta)u_{|\ext}$. Thus, we can easily compute
				$u|_{\ext}$ and $\ns u(x)$, for $x \in \ext.$
			We split the argument relatively to the value of $s$. 
            
            {\bf The case $0< s< 1/2$}. It is known that for such a value of $s$, the spaces $H^s(\Omega) = \widetilde H_0^s(\Omega)=H_0^s(\Omega)$ with equivalent norms (see e.g. \cite{Gris,War}). For $U\subset \rn$ an open set, we define the function $u_{U}$ by
            $$
            u_U(x) := \begin{cases}
                u(x) &\text{ if } x\in U,\\ 
                0 &\text{ otherwise.} \end{cases} $$ 
                Thus, if $u\in H^s(\rn)$, then  $u_{\Omega} \in \widetilde H_0^s(\Omega)$. Applying the linearity of the operator $\ns$ we obtain that
                     \begin{equation}\label{LND}
        \ns (u_\Omega) = \ns u - \ns (u_{\ext}).
        \end{equation}
Observe that the right-hand side of \eqref{LND} is known.
				Now, applying the definition of the operator $\ns$ and the fact that $u_{\Omega}$ vanishes in $\ext$, we obtain that for a.e. $x\in\ext$,
				$$
				\begin{aligned}
					(\ns u_{\Omega})(x) = C_{n,s}\int_\Omega \sfracf{u_\Omega(x)-u_\Omega(y)}\; dy =C_{n,s}\int_{\rn} \sfracf{u_\Omega(x) - u_{\Omega}(y)}\;dy 
					=  [\lfrac u_{\Omega}](x).
				\end{aligned}
				$$
				We can thus apply Theorem \ref{thm_grsu_tikonov} to reconstruct $u|_{\Omega}$. Since the right hand side of \eqref{q_robin_frac} is very smooth, then by the regularity obtained  by Ros-Oton \& Serra in \cite{RS-DP,ros2014extremal} we have that $u\in C^{0,2s+\varepsilon}(\Omega)$ for some $\varepsilon>0$, $(-\Delta)^su$ is a continuous function on $\Omega$ and $(-\Delta)^su(x)$ exists pointwise for all $x\in\Omega$ (see e.g. Kochubei et al. \cite[Theorem 2, pp 246]{HB-FC}).
                We can deduce that the first equation in \eqref{q_robin_frac} is satisfied pointwise for every $x\in\Omega$, so that, $q$ is given by \eqref{eq-q}.

                {\bf The case $1/2\le s<1$}. We do no longer have that $H^s(\Omega)$ and $\widetilde H_0^s(\Omega)$ coincide. But, assuming a bit more regularity on the Robin data we can still find a workaround. Let $\Omega_0$ be an open neighborhood of $\Omega$ such that $\overline \Omega \subset \Omega_0$, and suppose that $f, \RthLI f \in H^s(\Omega_0 \backslash \Omega)$. Suppose moreover that $\frac{1}{\theta'-\theta}$ is bounded and Lipschitz continuous in $\Omega_0 \backslash \Omega$.
We know that $u|_{\ext} = \frac{1}{\theta'-\theta} (\RthLI f - f)$. Thus, we can conclude that, under the previous assumptions, $u \in H^s(\Omega_0)$. We let $h \in H^s(\Omega_0)$ be any smooth extension of $u|_{\ext}$ inside $\Omega$. In other words $h|_{\ext} = u|_{\ext}$. Letting $v:= u-h$, we clearly see that $v\in \widetilde H_0^s(\Omega)$. Now, applying the same logic as in the previous case, we can easily compute $\ns v$ as $\ns u - \ns h$, which is known. Thus, $v$ can be reconstructed with the help of Theorem \ref{thm_grsu_tikonov}, and the reconstruction of $u$ follows immediately as $u=v+h$. Now, that $q$ is given by \eqref{eq-q} follows as in the first case by using the fact that in that case $u\in C^{1,2s-1+\varepsilon}(\Omega)$ for some $\varepsilon>0$, $(-\Delta)^su$ is a continuous function on $\Omega$ and
                $(-\Delta)^su(x)$ exists pointwise for all $x\in\Omega$ by \cite{HB-FC}. Hence, the first equation in \eqref{q_robin_frac} is satisfied pointwise for every $x\in\Omega$.
			\end{proof}
            	
        \subsubsection{The inverse problem for $L^1$-weak solutions}
            In this subsection, without any further mention, we assume that the coefficients $\theta, \theta' \in \Lambda_1$ are such that $(\theta-\theta') \in L^\infty(\ext)$ and $\theta-\theta' \neq 0$ a.e. in $\ext$, and we recall that $\Lambda_1$ has been defined in \eqref{Lambda}. We have the following result.
        
            \begin{lemma}
                Let $\theta, \theta' \in \Lambda_1$.             
                If $(\theta-\theta') \in L^\infty(\ext)$, then 
                $\wtt \simeq \wtp$
                with equivalent norms.
            \end{lemma}
            
            \begin{proof}
            Suppose without loss of generality that $\theta' > \theta$ a.e. in $\ext$. It is easy to see that we have the continuous embedding $\wtp\hookrightarrow \wtt$.
            
                Next, let $u \in \wtt$. Then, $\displaystyle\int_{\ext} |\theta| u^2\;dx < \infty$. 
                Besides, we also know that $$
                \int_{\ext} \theta' u^2\;dx = \int_{\ext} |\theta + (\theta'-\theta)|u^2\;dx \leq \int_{\ext} |\theta|u^2\;dx + \int_{\ext} |\theta'-\theta|u^2\; dx.
                $$
                Since $u\in L^2(\rn)$ and $(\theta-\theta') \in L^\infty(\ext)$, it follows that $\displaystyle \int_{\ext} |\theta'-\theta|u^2\;dx < \infty$. We can conclude that $\|u\|_{\wtp} < \infty$.
            \end{proof}

            \begin{definition}[\bf $L^1$-Robin-to-Robin map]
                Let $f\in L^2(\ext)$ and $u_f$ be the $L^1$-weak solution to the Robin problem \eqref{q_robin_frac} with given data $(f, \theta)$. We call the $L^1$-Robin-to-Robin map,  the mapping given by
			$$
			\RthLo : L^2(\ext) \rightarrow L^2(\ext),\qquad \ns u + \theta u \mapsto \ns u + \theta' u,
			$$
			which can be rewritten in a more simple way as
			\begin{equation*}
				\RthLo : L^2(\ext) \rightarrow L^2(\ext),\qquad f \mapsto  f + (\theta'-\theta)u_f.
			\end{equation*}
            \end{definition}
            
            We make the following observation.
            
            \begin{remark}
                The definition above is almost the same than the definition of the $L^\infty$-Robin-to-Robin map. Nevertheless, we emphasize that the $L^1$-Robin-to-Robin map is well defined thanks to the fact that $\wtt \simeq \wtp$ and that $u$ is also a weak solution in $\wtp$ of the problem
                $$
                \begin{cases}
					\lfrac u+ qu  = 0 \quad &\text{ in  } \Omega \\	
					\ns u + \theta' u = \RthLo f \quad &\text{ in  } \ext.
				\end{cases}
                $$
            \end{remark}

            The following theorem is the main result of this subsection.

            \begin{theorem}
				Let $\Omega \subset \rn$ be a nonempty bounded open set with a Lipschitz continuous boundary. Let $q\in C(\overline \Omega)$ be such that $q\ge q_0>0$ in $\Omega$ for some constant $q_0$ and let $u$ be the unique $L^1$-weak solution to the Robin problem \eqref{q_robin_frac}. Suppose that 
                \begin{itemize}
                \item $0<s<1/2$, or;
                \item $1/2\le s<1$, the Robin-to-Robin data $(f, \RthLo f)$ are of $H^{s}$ regularity in a neighborhood of $\Omega$, and that $\frac{1}{\theta'-\theta}$ is bounded and Lipschitz continuous near $\partial \Omega$. 
                \end{itemize}
                Then $u|_{\Omega}$ can be reconstructed from the knowledge of the unique pair $(f, \RthLo f)$, and $q$ is given as in \eqref{eq-q}.
			\end{theorem}
            
            \begin{proof}
                The proof is almost identical to the proof for $L^\infty$-weak solutions with the necessary modifications. We omit the details for the sake of brevity. 
            \end{proof}


\section{A new notion of Kelvin transform}\label{Kel-T}
In this section, we focus our attention on the Kelvin transform, an operator that is mostly featured in potential theory. The Kelvin transform performs a weighted sphere inversion, and swaps neighborhoods of the origin and neighborhoods of the infinity. Its most important property is that it preserves harmonic functions.  

The Kelvin transform was used by Riesz for the original proof of the UCP for Riesz Potentials (see e.g., Riesz \cite[Chap. III.11]{riesz} and \cite[Remark 3.2]{GRSU}). Following a similar idea, we use this operator to introduce a reconstruction algorithm in Section 5. But first, we define through duality a weak version of this operator that will enable us to deal with negative order Sobolev spaces. To the best of our knowledge, it is the first time that a weak Kelvin transform has been introduced. In this section $\alpha \in (0,2)$ is an arbitrary real number. 
        
        \begin{definition}
            Let $U \subset \rn$ be such that $0 \notin \overline U$. Then, we call the sphere inversion of $U$ and denote it by $U^\circ$, the set defined by
            $\displaystyle U^\circ := \left \{ \frac{y}{|y|^2}:\; y \in U\right \}.
            $
        \end{definition}

        Next, we introduce the classical Kelvin transform.
        
        \begin{definition}[\bf Strong Kelvin Transform]
        Let $f: \rn \rightarrow \mathbb R$ be a measurable function. We call the strong Kelvin transform of $f$ and denote it by $K_\alpha[f]$,  the function defined on $\rn_*:=\rn\setminus\{0\}$ by
        $$
        K_\alpha[f]: x \mapsto \frac{1}{|x|^{n-\alpha}}f\left (\frac{x}{|x|^2} \right ).
            $$
        \end{definition}

        We prove a few regularity results that we need concerning the strong Kelvin transform.
        
            \begin{proposition}\label{prop_kelvin_regularity}
           Let $0< r < 1$ be a real number and $W\subset\rn$ a bounded open set such that $0 \notin \overline W$. Let $\varphi\in \widetilde H_0^r(W)$. Then, $K_\alpha[\varphi] \in \widetilde H_0^r(W^\circ)$ and there is a constant $C(W,n,\alpha)>0$ such that 
           \begin{equation}\label{SKT}
           \|K_\alpha[\varphi]\|_{H^r(\rn)} \leq C(W, n, \alpha) \|\varphi\|_{H^r(\rn)}.
           \end{equation}
           Moreover, if $\varphi \in \mathcal D(\rnstar)$, then $K_\alpha[\varphi] \in 
           \mathcal D(\rnstar)$.
        \end{proposition}
        
    \begin{proof}
       Let $\varphi\in \widetilde H_0^r(W)$.  First of all, applying the change of variables $y:=\frac{x}{|x|^2}$ which Jacobian is given by $\displaystyle\frac{1}{|y|^{2n}}$, we obtain that
        $$
        \int_{\rn} |K_\alpha[\varphi]|^2\;dx = \int_{\rn} \frac{1}{|x|^{2n-2\alpha}} \left|\varphi\left(\frac{x}{|x|^2}\right)\right|^2\,dx = \int_{W} |y|^{-2\alpha}|\varphi(y)|^2\,dy.
    $$
    Since $\text{supp}(\varphi) \subset \overline W$ and $\overline W$ is away from the origin, we have that there is a constant $C_W>0$ such that
    $$\|K_\alpha[\varphi]\|_{L^2(\rn)} \leq C_W \|\varphi\|_{L^2(\rn)}.$$
    Now, regarding the other part of the norm, we observe that
    \begin{align}\label{eq-4-4}
    [K_\alpha[\varphi]]_{H^r(\rn)}^2 &:= \int_{\rn}\int_{\rn} \frac{\left (K_\alpha[\varphi](x)- K_\alpha[\varphi](y)\right )^2}{|x-y|^{n+2r}}\;dy\;dx \notag\\&= \iint_{(\rn)^2\backslash (\complement W^\circ)^2} \frac{\left (K_\alpha[\varphi](x)- K_\alpha[\varphi](y)\right )^2}{|x-y|^{n+2r}}\;dy\;dx,
    \end{align}
    with $\complement W^\circ := \rn \backslash W^\circ$. We know that $(\rn)^2\backslash (\complement W^\circ)^2 = (W^\circ \times W^\circ) \cup (\complement W^\circ \times W^\circ) \cup (W^\circ \times \complement W^\circ)$. We have used the notation $[\cdot]_{H^r(\rn)}$ to mean that it is a semi-norm.
    Since the integrand in \eqref{eq-4-4} is clearly symmetric, we have that
\begin{align}\label{eq-4-5}
[K_\alpha[\varphi]]_{H^r(\rn)}^2 = &\int_{W^\circ}\int_{W^\circ} \frac{\left (K_\alpha[\varphi](x)- K_\alpha[\varphi](y)\right )^2}{|x-y|^{n+2r}}\;dy\;dx \notag\\
&+ 2\int_{\complement W^\circ}\int_{\times W^\circ} \frac{\left (K_\alpha[\varphi](x)- K_\alpha[\varphi](y)\right )^2}{|x-y|^{n+2r}}\;dy\;dx.
\end{align}
For the first integral in the right hand side of \eqref{eq-4-5}, we operate a classical separation scheme that can be found in  \cite[Lemma 5.3]{NPV}. We use the notation $\bar x:= x/|x|^2$. Then,
$$
\begin{aligned}
\int_{W^\circ} \int_{W^\circ} \frac{\left (K_\alpha[\varphi](x)- K_\alpha[\varphi](y)\right )^2}{|x-y|^{n+2r}}\;dy\;dx
&=  \int_{W^\circ} \int_{W^\circ} \frac{\left (\frac{\varphi(\bar x)}{|x|^{n-\alpha}}- \frac{\varphi(\bar y)}{|y|^{n-\alpha}}\right )^2}{|x-y|^{n+2r}}\;dy\;dx
\\&= \int_{W^\circ} \int_{W^\circ} \frac{\left (|\bar x|^{n-\alpha}\varphi(\bar x)- |\bar y|^{n-\alpha}\varphi(\bar y)\right )^2}{|x-y|^{n+2r}}\;dy\;dx
\\&\leq 2\left ( \int_{W^\circ} \int_{W^\circ} \frac{\left (|\bar x|^{n-\alpha}\varphi(\bar x)- |\bar x|^{n-\alpha}\varphi(\bar y)\right )^2}{|x-y|^{n+2r}}\;dy\;dx \right .\\& \quad\left . +\int_{W^\circ} \int_{W^\circ} \frac{\left (|\bar x|^{n-\alpha}\varphi(\bar y)- |\bar y|^{n-\alpha}\varphi(\bar y)\right )^2}{|x-y|^{n+2r}}\;dy\;dx \right )
\\ &\leq 2\left ( \int_{W^\circ}\int_{W^\circ} \frac{\left(\varphi(\bar x) - \varphi(\bar y) \right)^2 |\bar x|^{3n-2\alpha+2r}|\bar y|^{n+2r}}{|\bar x- \bar y|^{n+2r}}\,dy\;dx\right .
\\& \quad\left . +\int_{W^\circ} \int_{W^\circ} \frac{(|\bar x|^{n-\alpha}\varphi(\bar y)- |\bar y|^{n-\alpha}\varphi(\bar y))^2(|\bar x||\bar y|)^{n+2r}}{|\bar x - \bar y|^{n+2r}}\;dy\;dx \right ).
\end{aligned}
$$
    In the last inequality we have used the chordal formula
    $\displaystyle
    |x-y| = \frac{|\bar x - \bar y|}{|\bar x||\bar y|}$.
    Then, applying the same change of variables as for the $L^2$-norm, we obtain that
    \begin{align}\label{eq-align}
   & \int_{W^\circ} \int_{W^\circ} \frac{\left (K_\alpha[\varphi](x)- K_\alpha[\varphi](y)\right )^2}{|x-y|^{n+2r}}\;dy\;dx \notag\\
   \leq 2&\left ( \int_{W}\int_{W} \frac{\left(\varphi(\bar x) - \varphi(\bar y) \right)^2 |\bar x|^{n-2\alpha+ 2r}|\bar y|^{-n+2r}}{|\bar x- \bar y|^{n+2r}}\,d\bar{y}\;d\bar{x}\right .
            \notag\\
            & \left . +\int_{W} \int_{W} \frac{(|\bar x|^{n-\alpha}\varphi(\bar y)- |\bar y|^{n-\alpha}\varphi(\bar y))^2(|\bar x||\bar y|)^{-n+2r}}{|\bar x - \bar y|^{n+2r}}\;d\bar y\;d\bar x \right ).
    \end{align}
    We notice that the domain of integration is shifted from $W^\circ$ to $W$. In the following, we denote by $m$ and $M$, respectively, the lower and upper bounds of $W$ in the Euclidian norm. In the first integral in the right hand side of \eqref{eq-align}, the term $\left (|\bar x|^{n-2\alpha+ 2r}|\bar y|^{-n+2r} \right )^{n-\alpha}$ is bounded by $\max\{m^{n-2\alpha + 2r}, M^{n-2\alpha + 2r}\} \cdot \max\{m^{-n+2r}, M^{-n+2r}\} $. The same applies to the term $(|\bar x||\bar y|)^{-n+2r}$ in the second integral. Moreover, the latter can be estimated as
    \begin{align*}    
       & \int_{W} \int_{W} \frac{(|\bar x|^{n-\alpha}\varphi(\bar y)- |\bar y|^{n-\alpha}\varphi(\bar y))^2(|\bar x||\bar y|)^{-n+2r}}{|\bar x - \bar y|^{n+2r}}\;d\bar y\;d\bar x \\
        &\leq C_W \int_{W} \int_{W} \frac{|\varphi(\bar y)|^2(|\bar x|^{n-\alpha}- |\bar y|^{n-\alpha})^2}{|\bar x - \bar y|^{n+2r}}\;d\bar y\;d\bar x\\
        & \leq C_W \left (\int_{W} \int_{B_\varepsilon(\bar x)\cap W} \frac{|\varphi(\bar y)|^2(|\bar x|^{n-\alpha}- |\bar y|^{n-\alpha})^2}{|\bar x - \bar y|^{n+2r}}\;d\bar y\;d\bar x \right . \\
        &\quad + \left . \int_{W} \int_{W\backslash B_\varepsilon(\bar x)} \frac{|\varphi(\bar y)|^2(|\bar x|^{n-\alpha}- |\bar y|^{n-\alpha})^2}{|\bar x - \bar y|^{n+2r}}\;d\bar y\;d\bar x\right ),
        \end{align*}
    where the constant $C_W>0$ comes from the term $(|\bar x||\bar y|)^{-n+2r}$, and $\varepsilon < \mathrm{dist}(W, \{0\})/2$. We know that
    \begin{equation}\label{RHS-I}
    \left | |x|^{n-\alpha} - |y|^{n-\alpha}\right | \leq \sup_{z\in B_\varepsilon(x)} (n-\alpha)|z|^{n-\alpha-1} \cdot |x-y|, \qquad \forall\; y \in B_\varepsilon(x),
    \end{equation}
    as a consequence of the Mean Value Theorem. The right hand side in \eqref{RHS-I} is bounded as a consequence of the choice of $\varepsilon$. Moreover, using \eqref{RHS-I} and applying Fubini's theorem, we get that 
    $$
    \begin{aligned}
        \int_{W} \int_{W\cap B_\varepsilon(\bar x)} \frac{|\varphi(\bar y)|^2}{|\bar x - \bar y|^{n+2r-2}}\;d\bar y\;d\bar x &= \int_W \int_W \chi_{\varepsilon}(|\bar x- \bar y|) \frac{|\varphi(\bar y)|^2}{|\bar x - \bar y|^{n+2r-2}}\;d\bar y\;d\bar x
        \\&= \int_W |\varphi(\bar y)|^2 \int_W \chi_{\varepsilon}(|\bar x- \bar y|) \frac{1}{|\bar x - \bar y|^{n+2r-2}}\;d\bar x \;d\bar y
        \\ &\leq C\|\varphi\|^2_{L^2(W)},
    \end{aligned}
    $$
    where the constant $C$ depends only on $n, r$ and $W$. As a consequence of the above computations, we have the inequality
    $$
    \int_{W^\circ} \int_{W^\circ} \frac{\left (K_\alpha[\varphi](x)- K_\alpha[\varphi](y)\right)^2}{|x-y|^{n+2r}}\;dy\;dx \leq C(n,\alpha, W) \|\varphi\|_{H^r(W)}^2.
    $$
    
   If we turn our attention to the second part of the semi-norm $[K_\alpha[\varphi]]_{H^r(\rn)}$ in \eqref{eq-4-5}, we point out that, since $\mathrm{supp}(K_\alpha[\varphi] )\subset W^\circ$, we get 
   $$
   \begin{aligned}    
       \iint_{(\complement W^\circ \times W^\circ)} \frac{\left (K_\alpha[\varphi](x)- K_\alpha[\varphi](y)\right)^2}{|x-y|^{n+2r}}\;dy\;dx &= \int_{\complement W^\circ}\int_{W^\circ} \frac{\left (K_\alpha[\varphi](y)\right)^2}{|x-y|^{n+2r}}\;dy\;dx
       \\&= \int_{\complement W^\circ} \int_{W^\circ} \frac{\left (\frac{\varphi(\bar y)}{|y|^{n-\alpha}}\right )^2 }{|x-y|^{n+2r}}\;dy\;dx 
       \\& = \int_{\complement W} \int_W \frac{|\bar x|^{n+2r}}{|\bar x|^{2n}} \frac{|\bar y |^{3n-2\alpha+2r}}{|\bar y|^{2n}} \frac{(\varphi(\bar y))^2}{|\bar x - \bar y |^{n+2r}}\;d\bar y\;d\bar x
       \\&= \int_{\complement W} \int_W \frac{|\bar y |^{n-2\alpha+2r}}{|\bar x|^{n-2r}} \frac{(\varphi(\bar y))^2}{|\bar x - \bar y |^{n+2r}}\;d\bar y\;d\bar x.
\end{aligned}
$$
We split the domain $\complement W$ between $B_\varepsilon(0)$ and $\complement W \backslash B_\varepsilon(0)$ where $\varepsilon$ is chosen such that $\varepsilon < \mathrm{dist}(W, \{0\})/2$.  Then,
       \begin{equation}\label{eq4.6}
       \begin{aligned}
       \int_{\complement W} \int_W \frac{|\bar y |^{n-2\alpha+2r}}{|\bar x|^{n-2r}} \frac{(\varphi(\bar y))^2}{|\bar x - \bar y |^{n+2r}}\;d\bar y\;d\bar x = &\int_{B_\varepsilon(0)} \int_W \frac{|\bar y |^{n-2\alpha+2r}}{|\bar x|^{n-2r}} \frac{(\varphi(\bar y))^2}{|\bar x - \bar y |^{n+2r}}\;d\bar y\;d\bar x \\+ &\int_{\complement W \backslash B_\varepsilon(0)} \int_W \frac{|\bar y |^{n-2\alpha+2r}}{|\bar x|^{n-2r}} \frac{(\varphi(\bar y))^2}{|\bar x - \bar y |^{n+2r}}\;d\bar y\;d\bar x.
       \end{aligned}
       \end{equation}
 Regarding the first integral in the right-hand side of \eqref{eq4.6}, the term $\frac{|\bar y|^{n-2\alpha + 2r}}{|\bar x - \bar y|^{n+2r}}$ is bounded and $\frac{1}{|x|^{n-2r}}$ is integrable over $B_\varepsilon(0)$. Hence, we obtain the estimate
  $$
  \int_{B_\varepsilon(0)} \int_W \frac{|\bar y |^{n-2\alpha+2r}}{|\bar x|^{n-2r}} \frac{(\varphi(\bar y))^2}{|\bar x - \bar y |^{n+2r}}\;d\bar y \;d\bar x \leq C_{W, \varepsilon} \int_{B_\varepsilon(0)} \frac{1}{|\bar x|^{n-2r}} \int_W |\varphi(\bar y)|^2\;d\bar y\;d\bar x. 
  $$
Regarding the second integral in \eqref{eq4.6}, we need to handle differently the case $2r < n$ and the case $2r \geq n$.
In the case where $2r < n$, the term $\frac{|\bar y|^{n-2\alpha + 2r}}{|\bar x|^{n-2r}}$ is bounded over $\complement W \backslash B_\varepsilon(0)$ as a consequence of the fact that $|x| \geq \varepsilon$, and thus we get
$$
\int_{\complement W \backslash B_\varepsilon(0)} \int_W \frac{|\bar y |^{n-2\alpha+2r}}{|\bar x|^{n-2r}} \frac{(\varphi(\bar y))^2}{|\bar x - \bar y |^{n+2r}}\;d\bar y\;d\bar x \leq C_{W, \varepsilon} [\varphi]^2_{H^r(\rn)}.
$$
  The case $2r \geq n$ is different, since $\frac{1}{|x|^{n-2r}}$ is no longer bounded over $\complement W \backslash B_\varepsilon(0)$. We point out that whenever $|x|\rightarrow \infty$, we have $\frac{|x|^{2r-n}}{|x-y|^{n+2r}} = O(\frac{1}{|x|^{2n}})$. Hence, we derive the estimate
  \begin{equation}
      \int_{\complement W \backslash B_\varepsilon(0)} \int_W \frac{|\bar y |^{n-2\alpha+2r}}{|\bar x|^{n-2r}} \frac{(\varphi(\bar y))^2}{|\bar x - \bar y |^{n+2r}}\;d\bar y\;d\bar x \leq C_{W,\varepsilon} \|\varphi\|_{L^2(\rn)}^2.
  \end{equation}
  In both cases we have that there is a constant $C>0$ depends only on $n, \alpha, r$ and $W$ such that
  \begin{equation}
      \int_{\complement W} \int_W \frac{|\bar y |^{n-2\alpha+2r}}{|\bar x|^{n-2r}} \frac{(\varphi(\bar y))^2}{|\bar x - \bar y |^{n+2r}}\;d\bar y\;d\bar x \leq C \|\varphi\|_{H^r(\rn)}.
  \end{equation}
       Gathering all the above estimates, we finally get the desired inequality \eqref{SKT}.
            \end{proof}

        \begin{lemma}\label{lemma_l1_kelvin}
            Let $U\subset\rn $ be a bounded open neighborhood of the origin $0$ and suppose that $u\in L^2(\rn)$. Then, the strong Kelvin transform $K_\alpha[u] \in L^1(U)$, hence $K_\alpha[u] \in L^1(\rn)$.
        \end{lemma}
        
        \begin{proof}
        First, notice that $K_\alpha[u] \in L^1(\rn\setminus U)$. Second, using the change of variable $y:= \frac{x}{|x|^2}$, we get
            $$
            \begin{aligned}
            \int_U K_\alpha[u](x)\,dx &:= \int_U \frac{1}{|x|^{n-\alpha}} u\left ( \frac{x}{|x|^2} \right ) \,dx 
            \\&= \int_{U^\circ} \frac{|y|^{n-\alpha}}{|y|^{2n}} u(y)\,dy = \int_{U^\circ} \frac{1}{|y|^{n+\alpha}} u(y)\,dy.
            \end{aligned}
            $$
            The last integral is finite since $\frac{1}{|y|^{n+\alpha}} \in L^2(\rn \backslash B_\varepsilon)$, for every $\varepsilon > 0$. Thus, we can deduce that\\ $K_\alpha[u]=\frac{1}{|x|^{n-\alpha}} u\circ J \in L^1(U)$, and $K_\alpha[u] \in L^1(\rn)$.
        \end{proof}
       
       From now on and throughout the rest of the paper we will take $\alpha := 2s$. The following result is similar to the classical Riesz result \cite[Section 11]{riesz}.

        \begin{proposition}\label{prop_lfrac_kelvin}
            Let $\varphi\in \mathcal D(\rnstar)$, $0<s<1$, and
                $\alpha:=2s \in (0,2)$ be such that $n>\alpha$. Then,
                \begin{equation}\label{eq_kelvin_lfrac}
                \lfrac (K_\alpha[\varphi]) = \frac{1}{|x|^{2\alpha}}K_\alpha[\lfrac \varphi].
                \end{equation}
                 \end{proposition}
            
            \begin{proof}
            We recall that the Riesz-potential $I_{\alpha}u$ of a 
            function $u$ is defined by
            $$
            I_{\alpha} u(x) = C\int_{\rn} \frac{u(y)}{|x-y|^{n-\alpha}}dy,
            $$
            where $C$ is a normalization constant depending on $\alpha$ and $n$ only.
            
            We use the notation $\displaystyle\bar x = \frac{x}{|x|^2}$. The chordal distance formula yields
            $\displaystyle
            |\bar x - \bar y| = \frac{|x-y|}{|x||y|}.$ 
            Let $\varphi\in \mathcal D(\rnstar)$ and  set $v:=\lfrac \varphi$. Letting $\displaystyle v(\bar x):= \frac{1}{|\bar x|^{n+\alpha}}v(x)$ and using the above chordal formula, we obtain that
            $$
            I_{\alpha} v(x) = C\int_{\rn} \frac{v(y)}{|x-y|^{n-\alpha}}\;dy = C\int_{\rn} \frac{(|\bar y|^{n+\alpha}\bar v(\bar y))(|\bar x||\bar y|)^{n-\alpha}}{|\bar x- \bar y|^{n-\alpha}}\;dy, \qquad \forall \;x \neq 0.
            $$
            Applying the change of coordinates $\displaystyle\bar y = \frac{y}{|y|^2}$ with Jacobian $\frac{1}{|\bar y|^{2n}}$ we get
            \begin{equation}\label{eq_riesz_equality}
               I_{\alpha} v (x) = C\int_{\rn} \frac{(\bar v(\bar y))(|\bar x|)^{n-\alpha}}{|\bar x- \bar y|^{n-\alpha}}\;d\bar y = |\bar x|^{n-\alpha} I_{\alpha} \bar v(\bar x). 
            \end{equation}
            
            Now, in order to avoid any confusion as of the variable of study, we introduce the change of variable as the proper conformal mapping
            $\displaystyle
                 J:\rnstar \rightarrow \rnstar,\; x \mapsto \frac{x}{|x|^2}$.
            Thus,  \eqref{eq_riesz_equality} becomes
            $I_{\alpha}v = (|x|^{n-\alpha}I_{\alpha}\bar v)\circ J$,
            which can be rewritten as $
            I_{\alpha}v\circ J = |x|^{n-\alpha}I_{\alpha}\bar v$,
            since $J$ is an involution. Since the Riesz potential and the fractional Laplacian are inverse operators of each other, we have that
            $$
            \varphi \circ J = |x|^{n-\alpha} I_{\alpha} \bar v \implies \lfrac \left (\frac{1}{|x|^{n-\alpha}} \varphi\circ J\right ) = \frac{1}{|x|^{n+\alpha}} (\lfrac \varphi)\circ J.
            $$
            Finally, we obtain the equality $
            \lfrac(K_\alpha[\varphi]) = \frac{1}{|x|^{2\alpha}} K_\alpha[\lfrac \varphi]
            $.
            \end{proof}
            
            We are now able to introduce a generalization of the strong Kelvin transform that will help us to extend it to the space of distributions. To simplify the notations we shall further denote the duality $\langle\cdot,\cdot\rangle_{\mathcal{D}'(\rnstar), \mathcal D(\rnstar)}$ by simply $\langle\cdot,\cdot\rangle_{\ddualstar}$.
            
            \begin{definition}[\bf Weak Kelvin Transform]
            Let $u\in \mathcal D'(\rnstar)$. We call the weak Kelvin transform of $u$ and denote it by $K^*_\alpha[u]$, the operator given by
            $$
            \langle K^*_\alpha[u], \varphi \rangle := \langle u, \frac{1}{|x|^{2\alpha}} K_\alpha[\varphi]\rangle_{\ddualstar}, \qquad \varphi \in \mathcal  D(\rnstar).
            $$
            \end{definition}
            
            We have the following result.
            
    \begin{proposition}\label{prop_weak_kelvin_equivalence}
                Let $u\in L^1_{\rm loc}(\rnstar)$. If $K_\alpha[u] \in L^1_{\rm loc}(\rnstar)$, then
        \begin{equation}\label{eq_weak_kelvin_equivalence}
                \langle K^*_\alpha[u], \varphi \rangle = \langle K_\alpha[u], \varphi \rangle_{\ddualstar}, \quad\forall \varphi \in \mathcal D(\rnstar).
                \end{equation}
                Moreover, if $u \in L^2(\rn)$, then the equality \eqref{eq_weak_kelvin_equivalence} holds with $\mathcal D(\rnstar)$ replaced with $\mathcal{\overline S}_\alpha(\rn)$. 
            \end{proposition}
            
            \begin{proof}
                Since $K_\alpha[u] \in L^1_{\rm loc}(\rnstar)$, we can express the duality in its integral form, that is,
                $$
                \langle K_\alpha[u], \varphi \rangle_{\mathcal D^{'}_\star} = \int_{\rn} \frac{1}{|x|^{n-\alpha}}u\left(\frac{x}{|x|^2} \right) \varphi(x)\;dx.
                $$
                Using the change of variable $y := \frac{x}{|x|^2}$, the above integral becomes
                \begin{align*}
                \int_{\rn} u(y) \frac{|y|^{n-\alpha}}{|y|^{2n}} \varphi \left (\frac{y}{|y|^2} \right)dy &= \int_{\rn} u(y) \frac{1}{|y|^{2\alpha}} \left ( \frac{1}{|y|^{n-\alpha}} \varphi \left( \frac{y}{|y|^2} \right)  \right)\;dy\\
                &=\int_{\rn} u(y) \frac{1}{|y|^{2\alpha}}K_\alpha[\varphi](y)\;dy.
                \end{align*}
                Thus,
                $$
                \langle K_\alpha[u], \varphi \rangle_{\mathcal D'_\star} = \langle u, \frac{1}{|x|^{2\alpha}} K_\alpha[\varphi]\rangle_{\ddualstar}=\langle K_\alpha^*[u],\varphi\rangle.
                $$
                Now, suppose that $u\in L^2(\rn)$ and $\varphi \in \mathcal {\overline S}_\alpha(\rn)$. Since $\varphi$ is rapidly decreasing, we have that  the integral does converge at infinity. 
                On the other hand, we do not know how it behaves near the origin. Using Lemma \ref{lemma_l1_kelvin}, we know that $K_\alpha[u] \in L^1(U)$ for any neighborhood $U$ of the origin. Thus, the integral
                $\displaystyle
                \int_{\rn} \frac{1}{|x|^{n-\alpha}}u\left(\frac{x}{|x|^2} \right) \varphi(x)\;dx
                $
                is well defined. The same computations we did for the previous case leads us to the conclusion that
                $$
                \langle K_\alpha[u], \varphi \rangle_{\overline {\mathcal S}_\alpha(\rn)'} = \int_{\rn} u(y) \frac{1}{|y|^{2\alpha}} \left ( \frac{1}{|y|^{n-\alpha}} \varphi \left( \frac{y}{|y|^2} \right)  \right)\;dy.
                $$
    The right-hand side converges thanks to the decreasing property of $\varphi$. Indeed, using the notation $\bar y := \frac{y}{|y|^2}$ we obtain that
    $\displaystyle
     \frac{1}{|y|^{n+\alpha}} \varphi\left( \frac{y}{|y|^2} \right) = |\bar y|^{n+\alpha} \varphi(\bar y).
    $
    By the definition of $\overline{\mathcal S}_\alpha(\rn)$, $|\bar y|^{n+\alpha} \varphi(\bar y)$ is uniformly bounded. Since 
    $$\left ( \frac{\varphi(x/|x|^2)}{|x|^{n+\alpha}}\right)^2 \leq C_{\varphi}\frac{1}{|x|^{2n+2\alpha}},$$  
    and as $2n + 2\alpha > n$, it follows that the function $\frac{1}{|x|^{n+\alpha}}\varphi(x/|x|^2)$ belongs to $L^2(\rn)$, which gives us the desired convergence. 
            \end{proof}
            
    \begin{proposition}\label{prop_weak_kelvin_regularity}
        Let $W\subset \rnstar$ be a nonempty bounded open set such that $\overline W \cap \{0\} = \emptyset$. If $u\in \mathcal D'(\rnstar)$, then $K^*_\alpha[u] \in \mathcal D'(W)$. Moreover, if $u \in H^{-r}(W)$ with $0< r <1$,  then $K^*_\alpha[u] \in H^{-r}(W^\circ)$.
    \end{proposition}
    
    \begin{proof}
            First of all, $K^*_\alpha[u]$ is naturally linear. Indeed, let $\lambda \in \mathbb R$ and $\varphi_1, \varphi_2 \in \mathcal D(W)$. Then,
            $$
            \begin{aligned}
               \left \langle K^*_\alpha[u], \lambda\varphi_1 + \varphi_2\right\rangle &= \langle u, \frac{1}{|x|^{2\alpha}}K_\alpha[\lambda\varphi_1+\varphi_2]\rangle_{\ddualstar}
                \\&= \left \langle u, \frac{1}{|x|^{2\alpha}}\left [ \frac{1}{|x|^{n-\alpha}} (\lambda\varphi_1+\varphi_2)\circ J\right] \right \rangle_{\ddualstar}
                \\&= \lambda\langle u, \frac{1}{|x|^{2\alpha}} K_\alpha[\varphi_1]\rangle_{\ddualstar}+\langle u, \frac{1}{|x|^{2\alpha}} K_\alpha[\varphi_2]\rangle_{\ddualstar}\\
                &= \lambda \langle K^*_\alpha[u], \varphi_1 \rangle+ \langle K^*_\alpha[u],  \varphi_2\rangle.
            \end{aligned}
            $$
            Moreover, $K^*_\alpha[u]$ is bounded over $\mathcal D(W)$. Indeed, denoting $W^\circ := J(W)$ the reflection of $W$ through the sphere inversion, we know that $\frac{1}{|x|^{2\alpha}}$ is bounded over $W^\circ$ with a constant that depends only on $W$. The upper bound of $\overline{W^\circ}$ in measure is the inverse of the lower bound of $\overline W$ in measure, which is strictly positive, since $\overline W$ is compact in $\rnstar$. Thus, we get the inequalities
            $$
            |\langle K^*_\alpha[u], \varphi\rangle_{\mathcal D'(W)}| \leq \|u\|_{\mathcal D'(\rnstar)} \sup\limits_{x\in W^\circ}\left|\frac{1}{|x|^{2\alpha}} K_\alpha[\varphi](x) \right| \leq C_{W}\|u\|_{\mathcal D'(\rnstar)} \sup_{x\in W}|\varphi(x)|.
            $$
            
            Suppose now that $u \in H^{-r}(W)$ and $\varphi \in \mathcal D(W^\circ)$ with $r>0$.
            By definition,
            $$
            \langle K^*_\alpha[u], \varphi \rangle  = \langle u, \frac{1}{|x|^{2\alpha}} K_\alpha[\varphi]\rangle_{\ddualstar}.
            $$
            We know from Proposition \ref{prop_kelvin_regularity} that $K_\alpha[\varphi]\in \widetilde H_0^r(W)$. Using very similar arguments as in the proof of Proposition \ref{prop_kelvin_regularity}, and since $\frac{1}{|x|^{2\alpha}}$ is bounded and locally Lipschitz continuous over $W$, we can also deduce that $\frac{1}{|x|^{2\alpha}}K_\alpha[\varphi] \in \widetilde H_0^r(W)$.
 Thanks to this, we can conclude that
            $$
            |\langle K^*_\alpha[u], \varphi\rangle | \leq C_{n,\alpha, W}\|u\|_{H^{-r}(W)}\|\varphi\|_{H^r(\rn)},
            $$
            and the proof is finished.
                 \end{proof}

            Proposition \ref{prop_weak_kelvin_equivalence} shows that the weak Kelvin transform coincides with the strong Kelvin transform in some sense. What we would like on top of that, is a result similar to Proposition \ref{prop_lfrac_kelvin} showing that the information carried by the fractional Laplacian is conserved through the weak Kelvin transform,  even when the fractional Laplacian is only defined in a weak sense.

            \begin{proposition}\label{prop_weak_kelvin_lfrac}
                Let $u \in 
                L^2(\rn)$, $0<s<1$ and $\alpha := 2s$ be such that $n>\alpha$. Then, for all $\varphi \in \mathcal D(\rnstar)$, we have that
                \begin{equation}\label{eq_weak_kelvin_lfrac}
                \langle K^*_\alpha[\lfrac u], \varphi \rangle = \langle \lfrac K_\alpha[u], |x|^{2\alpha}\varphi \rangle_{\mathcal D'(\rnstar)}.
                \end{equation}
            \end{proposition}
    
    \begin{proof}
                Let $\varphi\in \mathcal D(\rnstar)$. By definition, we have the following equalities:
                $$
                \begin{aligned}
                    \langle K^*_\alpha[\lfrac u], \varphi\rangle &= \langle \lfrac u, |x|^{-2\alpha}K_\alpha[\varphi]\rangle_{\mathcal D'_\star}
                    \\ &= ( u, \lfrac \left [|x|^{-2\alpha} K_\alpha[\varphi]  \right] )_{L^2(\rn)}
                    \\&= ( u, \lfrac (K_\alpha[|x|^{2\alpha}\varphi]) )_{L^2(\rn)}
                    \\&= ( u, |x|^{-2\alpha} K_\alpha[\lfrac (|x|^{2\alpha}\varphi)] )_{L^2(\rn)}
                    \\&= \langle K_\alpha[u], \lfrac(|x|^{2\alpha}\varphi) \rangle_{\mathcal {\overline S}_\alpha'(\rn_\star)}
                    \\&= \langle \lfrac K_\alpha[u], |x|^{2\alpha }\varphi \rangle_{\mathcal D'(\rnstar)},
                \end{aligned}
                $$
                where we used Proposition \ref{prop_lfrac_kelvin} in the fourth line, and Proposition \ref{prop_weak_kelvin_equivalence} in the fifth line. 
            \end{proof}
                
			
			\section{Reconstruction algorithms and applications}\label{Sec-RAA}
            
            This part of the article is devoted to the introduction of two reconstruction algorithms related to the UCP. The main idea is that, since there is only one function $u$ corresponding to the data
            $
            (u|_U, [\lfrac u]|_U)
            $
            in an open set $U\subset\rn$, then these data should suffice to reconstruct $u$ in the whole space. We use the following strategies.
            \begin{itemize}
            \item The first strategy makes use of the Kelvin transform to convert the unknown domain of $u$ (outside the set $U$) into a bounded domain without losing the information of the fractional Laplacian. 

            \item The second strategy is based on Green's functions related to the homogeneous and nonhomogeneous fractional Poisson equations in the ball. 
            \end{itemize}
            Combining both strategies leads to a Fredholm equation of the first kind which can be solved through classical or regularization tactics. 
            
We then bring up two applications to these algorithms. First, an alternative approach to the inverse Robin problem tackled in Section 3, and second an application to dynamical systems.
            
\subsection{Reconstruction algorithms}

Here are the reconstruction algorithms we propose.

    \subsubsection{Reconstruction with Kelvin transform}
    Theorem \ref{thm_grsu_tikonov} allows us to reconstruct from exterior measurements a function which support lies in a bounded domain. In this case though, we do not have any assumptions on the boundedness of our function's support. Thus, the main idea of the next result is to perform a series of invertible transforms that allow us to deal with a function with a compact support, while preserving the information carried by the fractional Laplacian. Recall that $0<s<1$ and $\alpha:=2s$.

    \begin{theorem}\label{thm_reconstructio_kelvin}
                Let $n\geq 1$ and $\alpha \in (0,2)$ be such that $n> \alpha$. Let $W\subset \rn$ be a nonempty bounded open set and  $u\in H^s(\rn)$. Then, $u$ can be reconstructed by the simple knowledge of $u|_{W}$ and $[\lfrac u]|{_W}$.
            \end{theorem}
            
            \begin{proof}
                The UCP for the fractional Laplacian states that $u$ is the only function corresponding to the data $(u|_{W}, [\lfrac u]|_{W})$. We suppose without loss of generality that $W$ contains the open ball $B_1(0)$. The results remain coherent after translation and scaling.
                To avoid any problem at infinity after the Kelvin transform, we deal with a different function from $u$. Using the extension property in Definition \ref{ext-prop}, let $g \in H^s(\rn)$ be an extension of $u$ outside $B_1(0)$. The extension property is a constructive result, thus $g$ can actually be computed. We define the function $h:= u-g$.
                Naturally, $h|_{B_1(0)} \equiv 0$ and $h \in H^s(\rn)$. 
                Taking the weak Kelvin transform of $h$ we point out that according to Proposition \ref{prop_weak_kelvin_lfrac} both $K_\alpha[h]$ and $\lfrac K_\alpha[h]$ are known outside $B_1(0)$, but not inside. Moreover, since $h|_{B_1(0)}$ is null, it follows that $K_\alpha[h]$ vanishes outside the ball. Hence, we have a compactly supported function which fractional Laplacian is known outside its support. But we cannot apply Theorem \ref{thm_grsu_tikonov} just yet since we do not know whether $K_\alpha[h]$ has $H^s(\rn)$ regularity, especially around the origin. Thus, we use a convolution-deconvolution scheme.
                 
                Let $\psi$ be a compactly supported smooth function such that $B_1(0) \subset \text{supp}(\psi)$. 
                The function $\psi$ could be a positive definite smooth radial function for example. Explicit formulas of such functions can be found in Wu \cite{CompactRadial}. We then define $\tilde h := \psi \star K_\alpha[h]$.
                By the regularization property of convolution, $\tilde h$ is smooth. It is also compactly supported as being the convolution of two compactly supported functions. 

                Now, as stated before, $\lfrac K_\alpha[h]$ is known in $H^{-s}(\rn \backslash B_1(0))$ through the weak Kelvin transform $K^*_\alpha[\lfrac h]$ and according to Proposition \ref{prop_weak_kelvin_lfrac}. Using the properties of convolution we point out that
                $$
                \lfrac \tilde h(x) = \lfrac (K_\alpha[h] \star \psi)(x) = (\lfrac K_\alpha[h] \star \psi)(x) = \langle \lfrac K_\alpha[h], \psi(x-t)\rangle_{\mathcal D'(\rn)}.
                $$
                If $|x|$ is large enough, then the support of $\psi(x-t)$ is translated. Thus, $\text{supp}(\psi(x-t)) \cap B_1(0) = \emptyset$. This means that, for sufficiently large values of $|x|$, $\lfrac \tilde h(x)$ depends only on $\psi$ and known values of $\lfrac K_\alpha[h]$. 
                
                Following this intuition, let $U\subset \rn$ be an open set such that
                $$
                \overline U \cap \text{supp}(\tilde h) = \emptyset \text{ and supp}(\psi(x-t)) \cap B_1(0) = \emptyset, \qquad\forall x \in U.
                $$
                Then, $[\lfrac \tilde h]|_{U}$ is known and, according to the regularization result in Theorem \ref{thm_grsu_tikonov}, $\tilde h$ can be reconstructed from the knowledge of $[\lfrac \tilde h]|_{U}$. Our hope is that we can also reconstruct $K_\alpha[h]$ inside the ball. If we turn our attention to the frequencies domain, the Fourier transform of $\tilde h$ yields
                $
                    \mathcal F(\tilde h) = \mathcal F(\psi) \cdot \mathcal F(K_\alpha[h]).
                $
                Since $K_\alpha[h] \in L^1(\rn)$ by Lemma \ref{lemma_l1_kelvin}, it follows that its Fourier transform is well defined. Since $\psi$ is a compactly supported function, its Fourier transform does not vanish on an open set. Moreover, $\psi$ is positive definite, and thus the zeros of $\mathcal F(\psi)$ are discrete and isolated. Hence, $\mathcal F(K_\alpha[h])$ can be reconstructed almost everywhere. Taking the inverse Fourier transform, we can thus recover $K_\alpha[h]$. Finally, we apply the Kelvin transform to $K_\alpha[h]$, which yields $h$ itself. Recalling that
                $u:= h+g$,
                and since $g$ is known by construction, the proof is finished.         
            \end{proof}
            
\begin{remark}
    If one has the \textit{a priori} information that $u$ vanishes fast enough when $|x| \rightarrow \infty$, then, the convolution-de-convolution part can be skipped, since it was only meant to bypass the singularity at the origin generated by the Kelvin transform.
\end{remark}

\begin{remark}
    We treated the case where the function is known inside the unit ball for simplicity purposes. But these results can be generalized to any ball in a pretty straightforward fashion.
\end{remark}

    \subsubsection{Reconstruction with Green functions}
    An alternative approach can be used to recover exterior values of $u$. A very useful property of the fractional Laplacian is that explicit formulas are known for the fractional Poisson Kernel in the ball. We refer the reader to Bucur \cite{greenball} for a thorough analysis. We will only make use of the following results. Throughout this subsection, without any mention, for a real number $r>0$, we let $B_r:=B_r(0)$, the open ball in $\rn$ of center $\{0\}$ and radius $r$.

    \begin{theorem}\cite[Theorem 2.10]{greenball}\label{thm_poisson_nonhomogeneous}
        Let $r>0$, $g \in \mathcal L^1_s(\rn) \cap C(\rn)$ and 
        \begin{equation}\label{eq_poisson_kernel}
        u_g(x) := 
        \begin{cases}
            \displaystyle \int_{\rn\backslash B_r} P_r(y,x) g(y) \, dy & \text{ if } x \in B_r,   \\
             g(x)& \text{ if }x\in \rn\backslash B_r,
        \end{cases}
        \end{equation}
        where
        \begin{equation}\label{eq-P-alpha}
            P_r(y,x) := c_{n,s} \left( \frac{r^2 - |x|^2}{|y|^2-r^2}\right )^s \frac{1}{|x-y|^n}, \;\qquad (x, y) \in B_r \times (\rn \backslash B_r).
        \end{equation}
 Then, $u_g$ is the unique pointwise continuous solution of the Dirichlet problem
        \begin{equation}\label{eq_poisson_nonhomogeneous}
            \begin{cases}
                 \lfrac u = 0 &  \text{ in } B_r,\\
                 u = g & \text{ in } \rn \backslash B_r. 
            \end{cases} 
        \end{equation}
    \end{theorem}
    
    In the above theorem, the constant $c_{n,s}$ is not to be mistaken with $C_{n,s}$ from \eqref{CN} (see \cite{greenball} for more details). We also have a formula for the homogeneous Poisson problem. It uses the following Green function (\cite[Theorem 3.1]{greenball}):
    \begin{equation}\label{eq-G}
        G(x, z) := \begin{cases}
           \displaystyle \kappa(n,s) |z-x|^{2s-n}\int_{0}^{r_0(x, z)} \frac{t^{s-1}}{(t+1)^{n/2}} \, dt & \quad \text{ if } n\neq 2s\\
           \displaystyle \frac{1}{\pi} \log \left ( \frac{r^2 - xz + \sqrt{(r^2 - x^2) (r^2-z^2)}}{r|z-x|}\right ) & \quad \text{ if } n= 2s,
        \end{cases}
    \end{equation}
    where
    $$
    \begin{aligned}
        r_0(x,z) := \frac{(r^2-|x|^2)(r^2-|z|^2)}{r^2|x-z|^2} \text{\; and \;} \kappa(n,s) := \frac{\Gamma(\frac{n}{2})}{2^{2s}\pi^{n/2}\Gamma^2(s)}.
    \end{aligned}
    $$
    
     \begin{theorem}\cite[Theorem 3.2]{greenball}\label{thm_poisson_homogeneous}
        Let $r > 0, h \in C^{2s+\varepsilon}(B_r) \cap C(\overline B_r)$ for some $\varepsilon>0$, and 
        \begin{equation}
            u(x) := \left \{ \begin{array}{ll}
                \displaystyle \int_{B_r(0)} h(y) G(x,y) \, dy & \text{if } x \in B_r,  \\
                 0& \text{if } x \in \rn \backslash B_r.
            \end{array} \right .
        \end{equation}
        Then, $u$ is the unique pointwise continuous solution of the Poisson problem
        \begin{equation}\label{eq_poisson_homogeneous}
            \begin{cases}
                 \lfrac u =  h \quad & \text{ in } B_r,  \\
                 u =  0 &\mbox{ in }\rn\setminus B_r.
         \end{cases}        
        \end{equation}
    \end{theorem}
    
    The above results allow us to represent the value of $u$ outside $B_r$ as a solution to an integral equation.

    \begin{theorem}\label{thm_fredohlm}
        Let $r>0$ and $u \in C(\rn) \cap \mathcal{L}^1_s(\rn)$. Let $P_r$ and $G$ be defined as in \eqref{eq-P-alpha} and \eqref{eq-G}, respectively. Suppose that $u|_{B_r}$ and $[\lfrac u]|_{B_r}$ are known and that $[\lfrac u]|_{B_r} \in C^{2s+\varepsilon}(B_r) \cap C(\overline B_r)$ for some $\varepsilon>0$. Then, $g := u|_{\rn\backslash B_r}$ is the unique solution to the Fredholm equation of the first kind
        \begin{equation}\label{eq_fredholm_green}
            u(x) - u_h(x) = \int_{\rn \backslash B_r} P_r(x,y) g(y) \, dy, \qquad\forall x \in B_r,
        \end{equation}
        where $u_h$ is the unique solution of \eqref{eq_poisson_homogeneous} with $h := [\lfrac u]|_{B_r}$.
    \end{theorem}
    
    \begin{proof}
        First of all, we stress that a continuous function $u$ satisfying the assumptions of the present theorem can be decomposed as 
        $u(x) = u_h(x) + u_g(x)$,
        where $u_g$ is a solution of \eqref{eq_poisson_nonhomogeneous} with exterior value $g := u|_{\rn\backslash B_r}$. The function $u_h$ is known by Theorem \ref{thm_poisson_homogeneous}, and $u$ is known in $B_r$ by assumption.
        Henceforth, using Theorem \ref{thm_poisson_nonhomogeneous} we get
        $$u(x) - u_h(x) = u_g(x) = \int_{\rn \backslash B_r} P_r(x,y) g(y) \, dy, \quad \forall x \in B_r.$$
        For the uniqueness, we again use the UCP. 
        Suppose that there are two solutions $g_1, g_2 \in C(\rn) \cap \mathcal L^1_s(\rn)$ of \eqref{eq_fredholm_green}. Substracting both identities we get
        \begin{equation}
            \int_{\rn \backslash B_r} P_r(x,y) (g_1(y)-g_2(y)) \, dy = 0, \quad \forall x \in B_r.
        \end{equation}
        Now, let $u_{12}$ be defined as
        $$
        u_{12}(x):= \begin{cases}
            \displaystyle \int_{\rn\backslash B_r} P_r(y,x) (g_1-g_2)(y) \, dy & \text{ if } x \in B_r,   \\
             g_1(x) - g_2(x)& \text{ if }x\in \rn\backslash B_r. 
        \end{cases}
        $$
        By Theorem \ref{thm_poisson_nonhomogeneous}, we know that $u_{12}$ is the unique pointwise continuous solution to the system
        $$
        \begin{cases}
            \lfrac u_{12} = 0 \quad &\text{ in } B_r,\\
            u_{12} = g_1-g_2 \quad &\text{ in } \rn \backslash B_r.
        \end{cases}
        $$
        By design, we have that $[\lfrac u_{12}]|_{B_r} = u_{12}|_{B_r} = 0$, which by Lemma \ref{lemma_unique_continuation} implies that $u_{12} \equiv 0$. Thus, $g_1 - g_2 = 0$ and we get the desired uniqueness.
    \end{proof}
    
We also want to append a Tikhonov regularization result. For that purpose we first prove a few properties of the forward operator associated with the Poisson kernel $P_r$. 

\begin{theorem}
    Let $r>0$ and $g \in H^m(\rn)$ with $m > n/2 \geq m-1$.
    Let $\mathcal K_r$ be the linear integral operator associated with the Poisson kernel $P_r$ given in \eqref{eq-P-alpha}. Then,
    $$
    \mathcal K_r: H^m(\rn\backslash B_r) \rightarrow L^2(B_r),
    $$
    is an injective and compact operator with dense range.
\end{theorem}

\begin{proof}
    First of all, $\mathcal K_r$ is injective since we have shown in the proof of Theorem \ref{thm_fredohlm} that $\mathcal K_rg_1 = \mathcal K_r g_2$ implies that $g_1 = g_2$ in $\rn \backslash B_r$. Now for the compactness, we recall the boundedness of the operator. If $u_g$, for $g\in C(\rn) \cap \mathcal L^1_s(\rn)$, is defined as in \eqref{eq_poisson_kernel}, then the inequalities
    $$
    \begin{aligned}
\left|u_g(x)\right| & \leqslant \int_{R>|y|>r} P_r(y, x)|g(y)| d y+\int_{|y|>R} P_r(y, x)|g(y)|\;d y \\
& \leqslant c_{n,s} \sup _{y \in \bar{B}_R \backslash B_r}|g(y)|+2^{n+s} c_{n,s}\left(r^2-|x|^2\right)^s \int_{|y|>R} \frac{|g(y)|}{|y|^{n+2 s}}\;d y \\
& \leqslant c_{n,s} \sup _{y \in \bar{B}_R \backslash B_r}|g(y)|+2^{n+s} c_{n,s} r^{2 s} \int_{|y|>R} \frac{|g(y)|}{|y|^{n+2 s}}\;dy,
\end{aligned}
    $$
hold for all $x \in B_r$ (see e.g. \cite[Theorem 2.10]{greenball}).
The last part is bounded, since $u_g := \mathcal K_r(g) \in C^\infty(B_r) \cap L^\infty (B_r)$. The regularity is due to the smoothness of $P_r$. By the assumption $m > n/2 \geq m-1$, we have that $H^m(\rn)$ is continuously embedded in the space $C_b(\rn)$ of bounded continuous functions (see \cite[Theorem 4.12 Part I Case A]{adams2003sobolev}). Thus, it follows that all the above results still hold for $g \in H^m(\rn)$.
From the same embedding theorem we know that $H^m(\rn \backslash B_r)$ is continuously embedded in $C^{0, \lambda}(\overline{\rn \backslash B_r}), 0<\lambda < \min(1, m-\frac{n}{2})$. Using  Ros-Oton \& Serra \cite[Proposition 1.7]{ros2014extremal} we have that $\mathcal K_r(g) \in C^{0,\beta}(\overline{B_r})$ with $\beta = \min(s, \lambda)$, and there is a constant $C>0$ such that
$$
\|\mathcal K_r(g)\|_{C^{0,\beta}(\overline{B_r})} \leq C\|g\|_{C^{0,\lambda}(\overline{\rn \backslash B_r})} \leq C \|g\|_{H^m(\rn \backslash B_r)}.
$$
We just proved that $\mathcal K_r$ is bounded in $C^{0,\beta}(\overline{B_r})$. But in order to get the desired compactness result we need to prove boundedness in $H^{\beta'}(B_r)$ with $0<\beta' < \beta$. Letting $u_g := \mathcal K_r(g)$, we have the estimates  
$$
\begin{aligned}
[u_g]_{H^{\beta'}(B_r)}^2 &= \int_{B_r}\int_{B_r} \frac{|u_g(x) - u_g(y)|^2}{|x-y|^{n+2\beta'}}\,dx\;dy \\
&\leq \int_{B_r}\int_{B_r} \frac{C_r\|g\|^2_{H^m(\rn)}|(x) - (y)|^{2\beta} }{|x-y|^{n+2\beta'}}\, dx\;dy \\
&\leq C_r \|g\|_{H^m(\rn)}^2 \int_{B_r}\int_{B_r} \frac{1}{|x-y|^{n+2\beta'-2\beta}}\;dx\;dy. 
\end{aligned}
$$
Regarding the last term, we know that $\beta' < \beta$, hence $n+ 2\beta' - 2 \beta < n$,  and the last integral is finite. 
This leads to reformulate the last inequality as
\begin{equation}\label{eq_Kr_boundedness1}
[u_g]_{H^{\beta'}(B_r)}^2 \leq C_{n, r, s} \|g\|^2_{H^m(\rn)}. 
\end{equation}
Moreover, we have that $\|u_g\|_{L^2(B_r)}^2 = \int_{B_r} u_g^2dx \leq \|u_g\|_\infty^2 |B_r| $. Now having in mind that $\mathcal K_r$ is bounded in $L^\infty(B_r)$, we get that is a constant $C>0$ such that
\begin{equation}\label{eq_Kr_boundedness2}
    \|u_g\|_{L^2(B_r)}^2 \leq C\|g\|_{H^m(\rn)}^2.
\end{equation}
Combining \eqref{eq_Kr_boundedness1}-\eqref{eq_Kr_boundedness2} we obtain that
$
\mathcal K_r: H^m(\rn \backslash B_r) \rightarrow H^{\beta'}(B_r)
$
is bounded. Using the compact embedding $H^{\beta'}(B_r)\hookrightarrow L^2(B_r)$ (see e.g \cite[Theorem 1.4.3.2]{Gris}), we finally get the desired compactness of $\mathcal K_r$.

At last, for the range of the operator, we show that if $f\in L^2(B_r)$, then 
\begin{equation}\label{eq_hanh_banach_density}
(\mathcal K_r(g), f)_{L^2(B_r)} = 0, \;\;\forall g \in H^m(\rn) \implies f = 0,
\end{equation}
which ensures the density of the image according to the Hahn-Banach Theorem. Suppose first that $f \in L^2(B_r)$ is such that $(\mathcal K_r(g), f)_{L^2(B_r)} = 0$ for all $g \in C_c(\rn)$. We recall that $C^\infty(\overline \Omega)$ is dense in $L^2(\Omega)$. Thus, for any $\varepsilon > 0$, it is possible to find $f_\varepsilon \in C^\infty(\overline \Omega)$ such that $\|f_\varepsilon - f \|_{L^2(B_r)} < \varepsilon$.
We recall that since all functions are locally $s-$harmonic up to a small error (see e.g. \cite[Theorem 1.1]{s_harmonic}), we can find $u_\varepsilon \in H^s(\rn) \cap C^s(\rn)$ and $R > r$ such that
$$
\begin{cases}
    \lfrac  u_{\varepsilon} = f_{\varepsilon} &\text{ in } B_r,\\
    u_{\varepsilon} = 0 &\text{ in } \rn \backslash B_R,
\end{cases}
$$
and $\|f_\varepsilon - u_\varepsilon\|_{C(\overline B_r)} < \varepsilon$.
It is straightforward to show that 
$\displaystyle
\int_{B_r} (u_\varepsilon - f_\varepsilon)^2 \, dx \leq \varepsilon^2|B_r|
$.
Using the triangle inequality on the $L^2$ norm we get
\begin{equation}\label{eq_density_inequality1}
\|u_\varepsilon - f\|_{L^2(B_r)} < \varepsilon + \varepsilon \sqrt{|B_r|}. 
\end{equation}

Let us notice that since $u_\varepsilon$ vanishes outside $B_R$ we have that $u_\varepsilon \in C_c(\rn)$. We set $g_\varepsilon := u_\varepsilon|_{(\rn \backslash B_r)}$. Then,  $u_\varepsilon$ is a solution of \eqref{eq_poisson_nonhomogeneous} with data $g_\epsilon$. In other words $u_\varepsilon|_{B_r} = \mathcal K_r(g_\varepsilon)$. 
By assumption, we have that $(u_\epsilon, f)_{L^2(B_r)} = 0$.
Therefore, using \eqref{eq_density_inequality1}, we get
\begin{equation}\label{eq5-13}
    \int_{B_r} u_\varepsilon^2\;dx + \int_{B_r} f^2\;dx < \varepsilon^2\left ( 1 + \sqrt{|B_r|}\right )^2. 
\end{equation}
Letting $\varepsilon \rightarrow 0$ in \eqref{eq5-13} we obtain that $\|f\|_{L^2(B_r)} = 0$ which proves the density of $\mathcal K_r(C_c(\rn))$ in $L^2(B_r)$. Moreover,  we know that $H^m(\rn)$ has dense intersection with $C_c(\rn)$. Indeed, $C_c^\infty(\rn) \subset H^m(\rn)$ and is canonically dense in $C_c(\rn)$. Since $\mathcal K_r$ is bounded, we get that $\mathrm{Im}(H^m(\rn))$ is also dense in $L^2(B_r)$. 
\end{proof}

We add a reconstruction theorem similar to Theorem \ref{thm_grsu_tikonov} as a straightforward application of the Tikhonov regularization schemes that can be found in Colton \& Kress \cite[Theorems 4.13 to 4.15]{colton1998inverse}.

\begin{theorem}\label{thm_reconstruction_fredholm}
    Let $r>0$ and $g \in H^m(\rn \backslash B_r)$ with $m > n/2 \geq m-1$.
    Let $\mathcal K_r$ be the linear integral operator associated with the Poisson kernel $P_r$ given in \eqref{eq-P-alpha}. Then, $g$ can be reconstructed from the knowledge of $\displaystyle u_g := \mathcal K_r g$ in $B_r$ as the limit of 
    $g = \displaystyle \lim_{\alpha \rightarrow 0} g_\alpha, \; \alpha > 0$ in $H^m(\rn \backslash B_r$,
    where 
    $$
    g_\alpha := \arg \min_{h\in H^m(\rn \backslash B_r)} \left [ \|\mathcal K_rh - u_g\|^2_{L^2(B_r)} + \alpha \|h\|^2_{H^m(\rn \backslash B_r) }\right ].
    $$
\end{theorem}
\begin{remark}
    We chose to place ourselves in the framework of Sobolev spaces for this regularization result since the Hilbert spaces have a very well-established setting for Regularization Theory, especially with a Tikhonov scheme. However, this choice comes with a cost. Much regularity is needed for the above theorem to hold true. We would like to point out that some alternatives do exist if one wanted to work within the setting of Banach spaces. See \cite{schuster2012regularization, lorenz2008optimal, weidling2020optimal,  burger2004convergence, hein2009tikhonov} and the references therein for some quite interesting literature on the topic. It is also important to mention that some more in-depth analysis of the forward operator would be needed in order to fit in this setting.
\end{remark}
    
 \subsection{Applications} 
    We present here two simple applications of the reconstruction algorithms featured in the previous subsection. We did not attempt to be thorough in the applications that might arise from these new reconstruction results.
    
    \subsubsection{Inverse Robin problem with a Dirichlet-to-Robin map}\,
    
        In this first application we consider an alternative way to study the Robin problem \eqref{q_robin_frac}. The difference here is that we use the Dirichlet-to-Neumann map $\Lambda_q$ from \cite{GRSU} to represent the nonlocal normal derivative in the Robin condition.
        We first assume that the potential $q$ is such that,
    $$
    \text{if } \quad \lfrac u+qu = 0 \text{ in } \Omega, \text{ and } u_{|\ext} = 0, \quad \text{ then } u\equiv 0\;\text{ a.e. in } \rn.
    $$
    This is always the case, for example if $q$ satisfies Assumption \ref{Ass-q}.
Then,  we know from \cite{GRSU} that the Dirichlet problem
    \begin{equation}\label{eq_dirichlet_frac}
    \begin{cases}
        \lfrac u+ qu  = 0 \quad &\text{ in } \Omega \\	
        u = \phi  &\text{ in } \ext,
    \end{cases}
\end{equation}
			for $\phi\in H^s(\ext)$, has a unique weak solution $u \in H^s(\rn)$ satisfying
			\begin{equation}\label{eq_weak_dirichlet}
				\int_{\rn}(-\Delta)^{s/2} u(-\Delta)^{s/2}v\;dx + \int_{\Omega}quv\;dx = 0, \quad \forall \;v \in \widetilde H_0^s(\Omega).	\end{equation}

            In this way,  we can define the appropriate Dirichlet-to-Neumann map 
            \begin{equation*}
				\Lambda_q : H^s(\ext) \rightarrow H^{-s}(\ext),\qquad \phi \mapsto  [\lfrac u]|_{\ext}.
			\end{equation*}
            The discussion about the justification of such a non-local DN-map can be found in \cite{GSU,GRSU}. More precisely, the case where $\Omega$ has a Lipchitz continuous boundary is contained in \cite[Lemma 2.4]{GSU}, and the case of an open set with a $C^\infty$-boundary in \cite[Lemma A.1]{GSU}.
			
			We seek to know whether for $\theta,\theta'\in \mathbb R_+^*$ and $\Omega_2$ a subset of $\ext$, the knowledge of the pair
			$$
				\Big(\Lambda_q \phi(x) + \theta \phi(x),\Lambda_q \phi(x) + \theta' \phi(x)\Big), \quad  x \in \Omega_2 \subseteq \ext,
			$$
             can be sufficient to identify the potential $q$. This will be our Dirichlet-to-Robin map.
            
            Putting the problem this way, we may emphasize that we actually do not need to consider $\Omega_2 = \ext$ in order to obtain a direct identification. In fact, to the contrary of the results in \cite{GRSU} we do not need $f:=u|_{\ext}$ to have a compact support in $\ext$.  The Dirichlet-to-Robin map can be measured in a (nonempty) compact part of the exterior, without any further assumption on $u$.
            
            \begin{lemma}\cite[Proposition 5.1]{GRSU}\label{lemma_measure_ucp}
                Let $\Omega \subset \rn$ be a bounded open set and let $q$ satisfy Assumption \ref{Ass-q}. Let $\frac 14\le s<1$ and assume that $u\in H^s(\rn)$ satisfies
                $$
                \lfrac u+qu = 0 \text{ in } \Omega,
                $$
                in the sense of \eqref{eq_weak_dirichlet}.
                If $u|_{E} = 0$ for some measurable set $E \subset \Omega$ with Lebsgue measure $|E|>0$,  then $u\equiv 0$ in $\rn$.
            \end{lemma}

            We have the following result that solves the identification and reconstruction problems.
            
			\begin{theorem}\label{lemma_omega_uniqueness}
            Let $\Omega\subset\rn$ be a bounded domain with a Lipschitz continuous boundary.
            Let $\Omega_2 \subseteq \ext$ be a nonempty open set which can be bounded.
        Let $f_1,f_2 \in H^s(\Omega_2)$ and $\theta, \theta' \in \mathbb R^*_+$.       
        Let $q$ be a function over $\Omega$ satisfying $q\ge q_0>0$ a.e. in $\Omega$ for some constant $q_0$, and the following conditions:
        \begin{itemize}
            \item  $s\in [\frac{1}{4}, 1)$ and $q\in L^\infty(\Omega)$, or;
            \item $s\in (0,1)$ and $q \in C(\overline \Omega)$.
        \end{itemize}
        Suppose that there exists $u\in H^s(\rn)$ a weak solution of
        $$
        \lfrac u + qu = 0 \quad \text{ in  }\Omega,
        $$ 
        in the sense of \eqref{eq_weak_dirichlet}, satisfying the systems
\begin{equation}\label{eq_double_robin}
            \begin{cases}
                \lfrac u+ q u = 0 \quad &\text{ in } \Omega, \\
                \Lambda_{q} u + \theta u = f_1 \quad &\text{ in } \Omega_2,
            \end{cases}
            \quad \quad
            \begin{cases}
                \lfrac u+ q u = 0 \quad &\text{ in } \Omega, \\
                \Lambda_{q} u + \theta' u = f_2 \quad &\text{ in } \Omega_2.
            \end{cases}
        \end{equation}
    Then, $u$ is unique. Moreover, q can be reconstructed a.e in  $\Omega$ from the knowledge of 
    $
    (\Lambda_{q}u + \theta u, \Lambda_{q}u + \theta' u).
    $
    \end{theorem}
            
			\begin{proof}
				Let us assume that there are two functions $u_1, u_2 \in H^s(\rn)$ which are weak solutions of both systems in \eqref{eq_double_robin} with possibly distinct potentials $q_1$ and $q_2$, respectively. The systems in \eqref{eq_double_robin} yield the system
				$$
				\begin{cases}
					\lfrac u_1 + \theta u_1 = \lfrac u_2 + \theta u_2 \quad &\text{ in } \Omega_2,\\
					\lfrac u_1 + \theta'u_1 = \lfrac u_2 + \theta' u_2\quad &\text{ in } \Omega_2.
				\end{cases}
				$$
				Subtracting both lines of the above system we obtain that
				$$
				\begin{cases}
					u_1 - u_2 = 0 \quad &\text{ in  } \Omega_2,\\
					\lfrac u_1 - \lfrac u_2 = 0 \quad &\text{ in } \Omega_2.
				\end{cases}				 
				$$
				Using Lemma \ref{lemma_unique_continuation}, we obtain that $u_1-u_2 \equiv 0$ in $\rn$. Thus, $u_1 = u_2$ a.e. in $\rn$. Moreover, using Lemma \ref{lemma_measure_ucp} and the same strategy as given in \cite[Theorem 1]{GRSU} we get 
                \begin{equation}\label{eq_potential_reconstruction}
                q_1 = q_2 = - \frac{\lfrac u_1}{u_1}\quad\text{a .e. in } \Omega.
                \end{equation}
                Now, we point out that the Dirichlet-to-Robin map gives us a direct knowledge of 
                $
                (u|_{\Omega_2}, [\lfrac u]|_{\Omega_2}).
                $
                Applying Theorem \ref{thm_reconstructio_kelvin} or, alternatively Theorem \ref{thm_reconstruction_fredholm}, if some strong regularity assumptions can be made about $u_1$ and $(f_1, f_2)$, then $u_1$ can be reconstructed in all $\rn$ and thus $q$ can be computed from  \eqref{eq_potential_reconstruction}
.                \end{proof}	

\subsubsection{Space-fractional heat equation}\,

We give here another application of the reconstruction schemes featured in this section. The space-fractional heat equation is a well studied model. Let $\Omega\subset \rn$ be a nonempty open set and $B$ an arbitrarily small nonempty region of $\Omega$. The following system is called the homogeneous fractional heat equation:
\begin{equation}\label{eq_fractional_heat_equation}
    \left \{\begin{array}{ccrccl}
         \partial_t u(x,t) &+& \lfrac u(x,t) &=& 0&   \quad \text{in } \Omega\times(0,\infty),\\
         & &u(x,t) & = & 0& \quad \text{in } \ext\times (0,\infty).
    \end{array} \right .
\end{equation}
Suppose that, through some kind of measurement device, one has knowledge of the solution $u$ of \eqref{eq_fractional_heat_equation} over $B \times [T_1, T_2]$, where $0\leq T_1 < T_2$. Then, one can reconstruct the solution $u$ in the domain $\Omega \times (T_1, T_2)$. The strategy is quite simple. Since $u|_{B}$ is known over $[T_1, T_2]$ then $\partial_tu$ can be computed, and using \eqref{eq_fractional_heat_equation}, one can recover $[\lfrac u]|_{B}$ over $(T_1, T_2)$. 
\begin{figure}[h]
    \centering
    \includegraphics[width=0.5\linewidth]{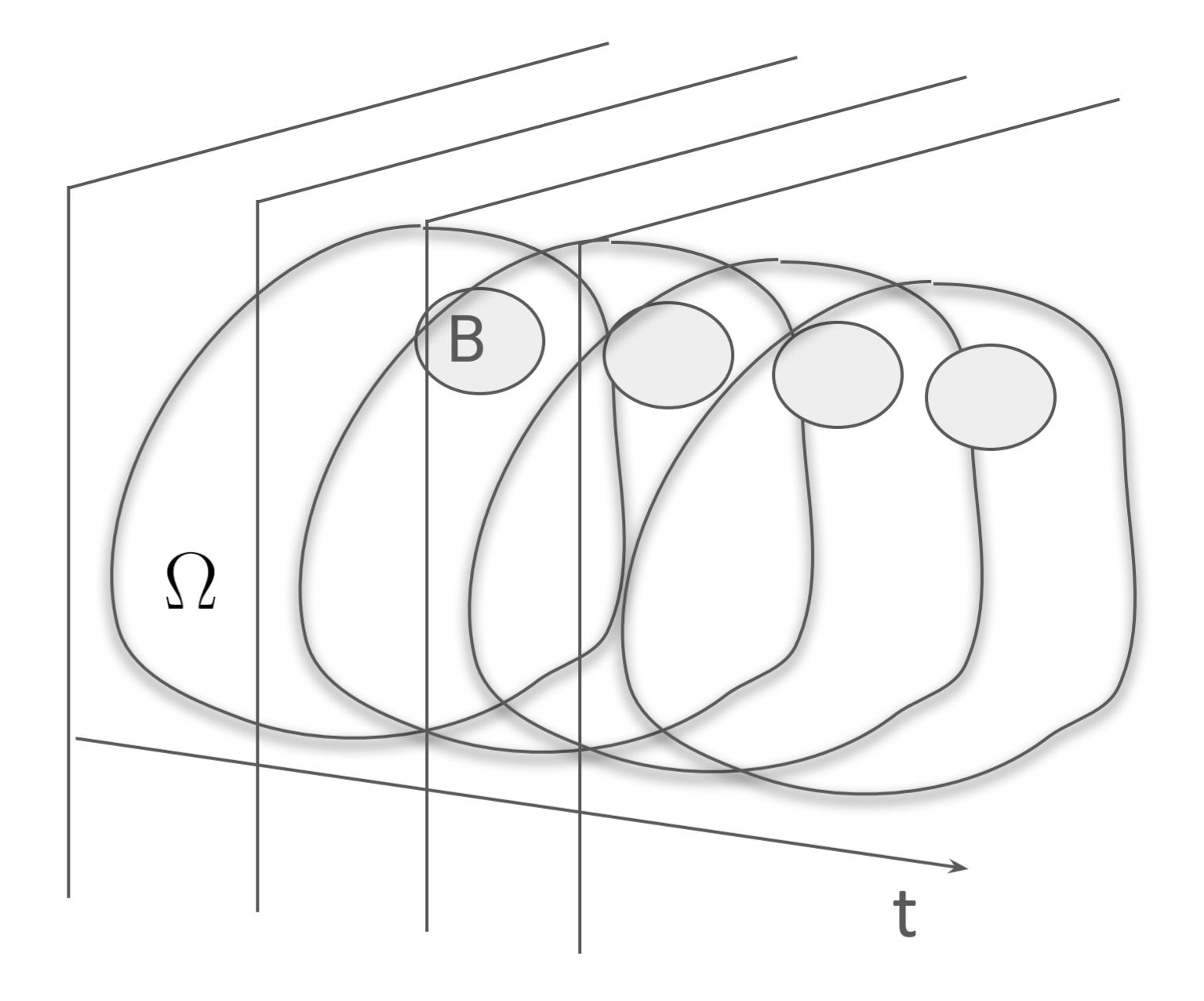}
    \caption{Fractional Heat Equation}
    \label{fig:placeholder}
\end{figure}
Thus, $(u|_B, [\lfrac u]|_{B})$ is known and $u$ can be computed in the whole space. 



\section{Numerical experiments}\label{Sec-Numerical}
In this section we present some numerical simulations results in 1D with $\Omega=(-1,1)$ for Theorem \ref{thm_reconstructio_kelvin}. We use the finite difference method developed by Duo, Wyk \& Zhang \cite{trapezoidalfl} to implement the discretizations of the fractional Laplacian. In order to deal with the Tikhonov optimization problem in Theorem \ref{thm_grsu_tikonov}, we choose a Krylov subspace method. The norm $\|\cdot\|_{H^s(\mathbb R)}$ in the penalization term \eqref{PT} is approximated through the finite difference matrix of the fractional Laplacian, using the semi-norm identity
$\displaystyle
[u]^2_{H^s(\mathbb R)} = \frac{2}{C_{1,s}}\|\lfrac u \|_{L^2(\mathbb R)}^2
$
(see e.g. \cite[Proposition 3.6]{NPV} for further details). Using the same notation as in Theorem \ref{thm_reconstructio_kelvin}, we let $h:= u-u_{ext}$ where $u_{ext}$ is an $H^s$-extension of $u$ outside the interval $(-1, 1)$.

We first need to implement this extension and for that purpose we use the constructive result of \cite[Theorem 5.4]{NPV}. This result implies the creation of two functions $\psi_1$ and $\psi_2$ that have support in a neighborhood of the boundary. For the functions $\psi_1$ and $\psi_2$  we choose mirrored smooth plateau functions of the form
$$\displaystyle
g(x) := \frac{f(x)}{f(x) - f(1-x)},
$$
where $f(x):= e^{-\frac{1}{x}}$. The mirroring ensures they have compact support, and they are translated so that $\psi_1$ is centered in $-1$ and $\psi_2$ is centered in $1$.
\begin{figure}[h]
    \centering
    \includegraphics[width=1\linewidth]{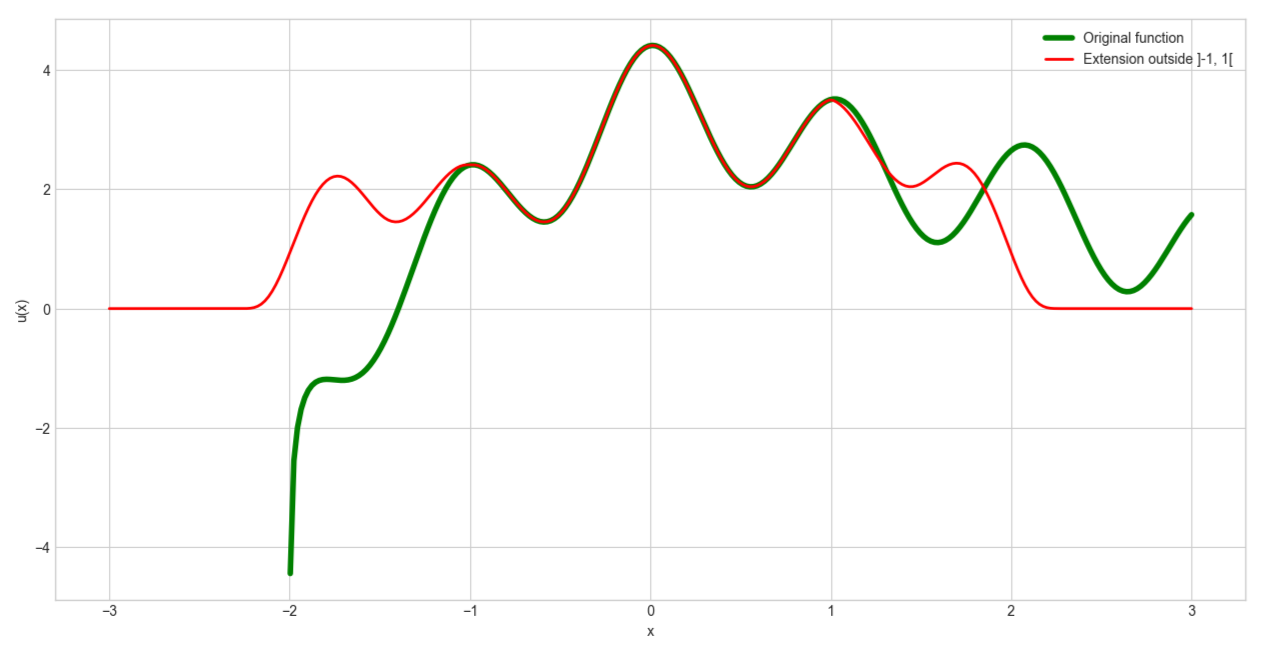}
    \caption{Smooth extension outside the ball}
    \label{fig:placeholder2}
\end{figure}

It is worth mentioning that we did not pursue any analysis of stability of the reconstruction beforehand. However, logarithmic stability results can be found in \cite{ruland2021single} for the Tikhonov reconstruction.

At a numerical level, differences appear between the fractional Laplacian of the Kelvin transform computed numerically and the values computed through $\frac{1}{|x|^{2\alpha}} K_\alpha[\lfrac h]$ given in \eqref{eq_kelvin_lfrac}. In essence, the shape of the fractional Laplacian does not vary between those two methods of computation. But, there are still noticeable discrepancies, especially near the spikes. The differences increase drastically as $s \rightarrow 1$.
An interpretation of this phenomenon may reside in truncation errors and regularity concerns. While the fractional Laplacian of the initial function, $\lfrac h$ is computed within some interval $(-B,B)$ with unavoidable truncation errors, the Kelvin transform is not known within the interval $(-\frac{1}{B}, \frac{1}{B})$, leading to unavoidable interpolation errors. Moreover, the finite-difference methods used to compute the fractional Laplacian (see e.g. \cite{trapezoidalfl,huang2014numerical}) are sensitive to the regularity of the function. It is indeed important to mention that the Kelvin transform creates spikes and singularities near the origin due to the fact that it contracts the space.

The following tables highlight these discrepancies. Note that we take $h := u-u_{ext}$ which vanishes inside the unit ball and thus makes the computation of $\lfrac K_\alpha[h]$ insensitive to truncation errors.

\begin{minipage}{0.43\textwidth}
\textbf{Table 1.}\\
\small{$u(x):=e^{-x^2}$}\\
\vfill
\begin{tabular}{c|cc}
\hline
$\alpha := 2s$ & Relative $L^2$ error & Absolute error \\
\hline
0.2 & 0.0226 & 0.0512 \\
0.6 & 0.0536 & 0.2177 \\
1 & 0.1188 & 0.9754 \\
1.5 & 0.2838 & 6.5153 \\
1.9 & 0.4850 & 29.5304 \\
\hline
\end{tabular}
 
\end{minipage}
\hfill 
\begin{minipage}{0.43\textwidth}
\textbf{Table 2.}\\
\small{$u(x):=\sin(x)/x$}\\
\vfill
\begin{tabular}{c|cc}
\hline
$\alpha := 2s$ & Relative error ($L^2$) & Absolute error $(L^\infty)$ \\
\hline
0.2 & 0.0442 & 0.4361 \\
0.6 & 0.1066 & 2.3095 \\
1 & 0.2057 & 12.1302 \\
1.5 & 0.3969 & 80.4332 \\
1.9 & 0.6056 & 342.5413 \\
\hline
\end{tabular}
\end{minipage}

\begin{figure}
\centering
\begin{minipage}{0.48\textwidth}
          \centering
          \includegraphics[width=1\linewidth]{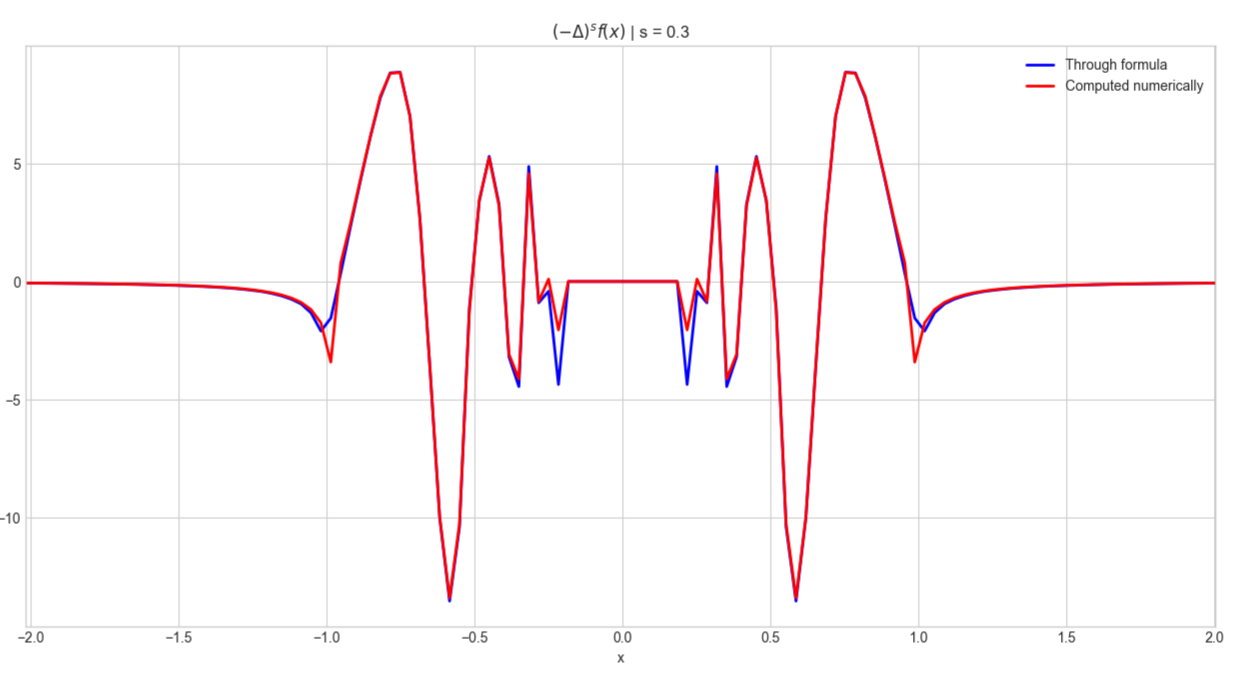}
\end{minipage}\hfill
\begin{minipage}{0.48\textwidth}
          \centering
          \includegraphics[width=1\linewidth]{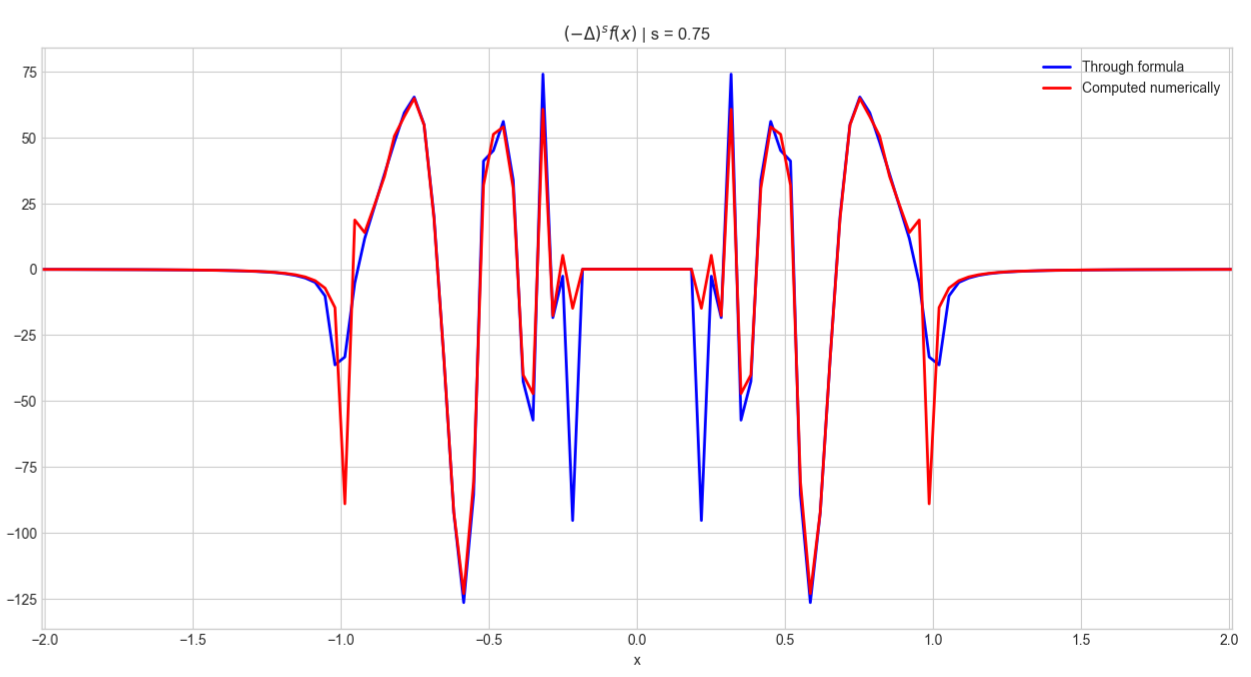}
      \end{minipage}
  \caption{$\lfrac K_\alpha[h]$ vs. $\frac{1}{|x|^{2\alpha}} K_\alpha[\lfrac h]$ for $s=0.3$ (on the left) and $s=0.75$ (on the right).}
\end{figure}


Outside $(-1, 1)$ the data behave better, since the Kelvin transform of $h$ has a compact support in the ball. Even so, the reconstruction has proven to be very sensitive to noise and prone to oscillations. 
\begin{table}
\centering
\begin{tabular}{c|ccccc}
\hline
$N \backslash \alpha$ & 0.2 & 0.6 & 1 & 1.5 & 1.9 \\
\hline
100 & 9.8586 & 5.8415 & 8.1391 & 20.4125 & 88.9307 \\
150 & 5.0618 & 3.4030 & 5.4128 & 15.3030 & 69.1308 \\
300 & 2.1148 & 1.5349 & 2.7948 & 9.7010 & 51.7716 \\
420 & 0.9620 & 0.4355 & 0.5287 & 1.8835 & 9.2894 \\
600 & 1.8909 & 0.4208 & 0.4695 & 1.4549 & 8.2669 \\
\hline
\end{tabular}
\caption{Relative $L^2$ errors in reconstruction}
\end{table}

\begin{figure}
    \centering
    \includegraphics[width=1\linewidth]{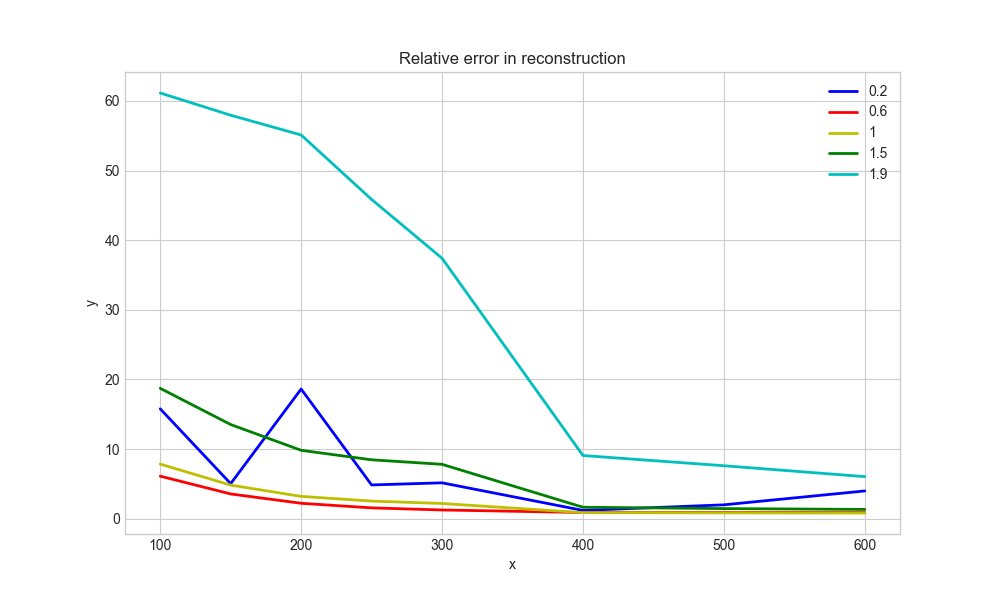}
    \caption{Relative $L^2$ errors in reconstruction for $h$ with $u(x) = \frac{\sin(x)}{x}$ and truncation interval is $(-5, 5)$. Each curve represents a different value of $\alpha := 2s$.}
    \label{fig_error_alpha}
\end{figure}
As it can be seen in Figure \ref{fig_error_alpha}, the error stabilizes when $N$ becomes large ($\gtrsim 400$). Reconstruction becomes almost impossible when $s \rightarrow 1$. It can be seen as a result of the fact that when $s\rightarrow 1$, the fractional Laplacian converges to the classical Laplacian in the sense of $H^{-s}(\rn)$, which means it approximates a local operator which does not enjoy the UCP.

\begin{figure}
    \centering
    \includegraphics[width=1\linewidth]{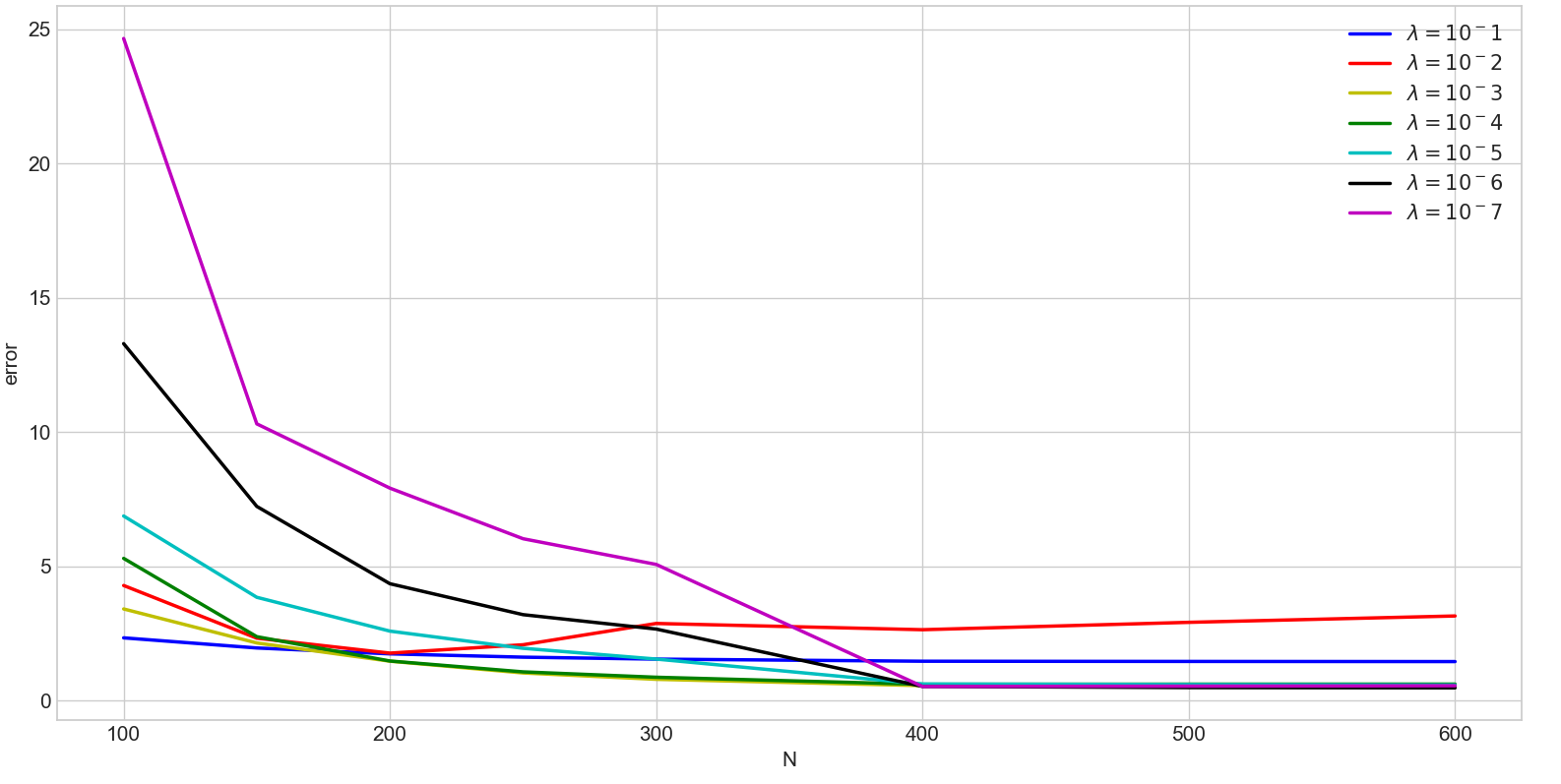}
    \caption{Relative $L^2$ errors in reconstruction for $h$ with $u(x) = e^{-x^2}$ where truncation interval is $(-5, 5)$ and $\alpha=0.5$. Each curve represents a different regularization coefficient $\lambda$.}
    \label{fig_error_regularization}
\end{figure}

\begin{figure}
    \centering
    \includegraphics[width=1\linewidth]{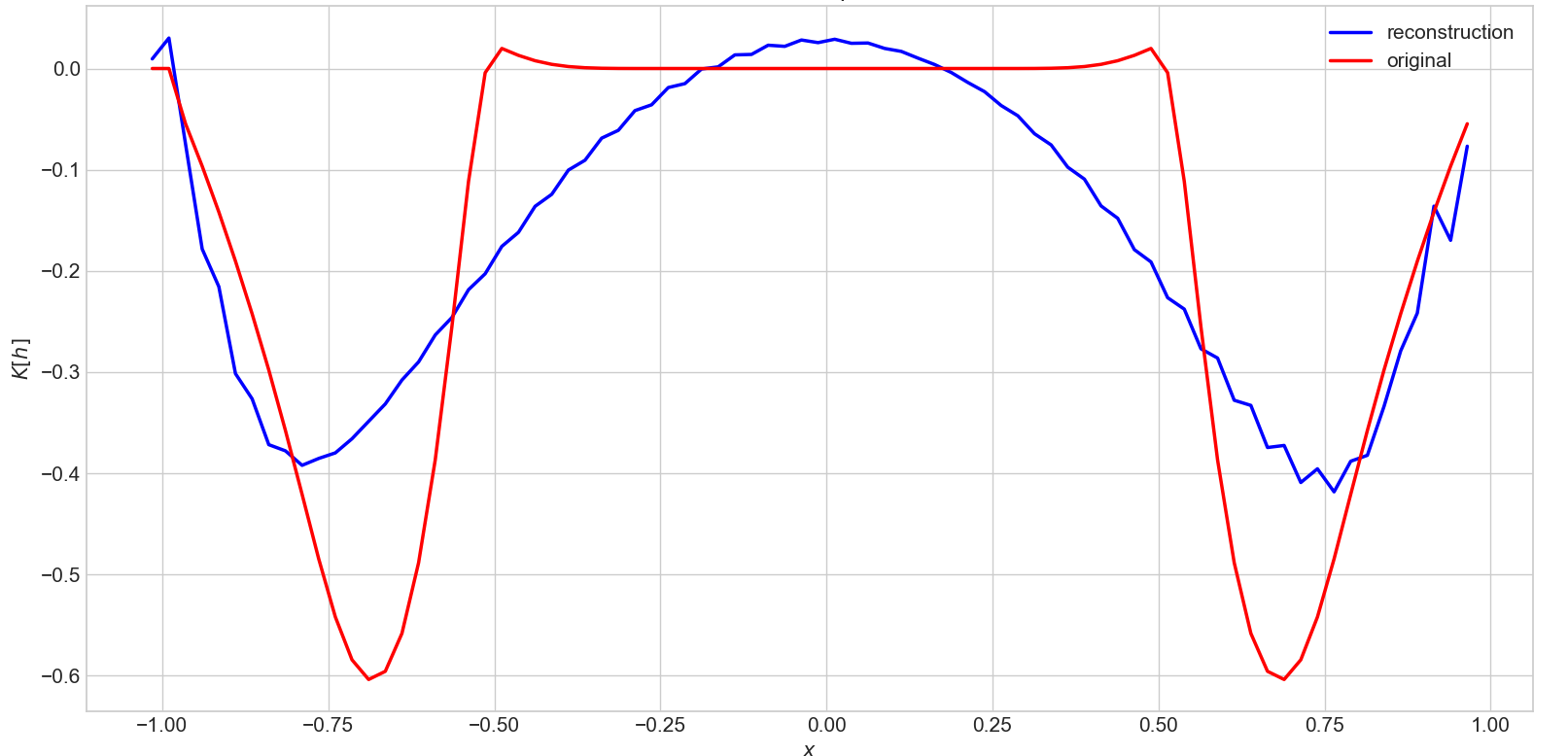}
    \caption{Reconstruction of the Kelvin transform of $h$ with $u(x) = e^{-x^2}$, $\alpha=0.6$ and $\lambda = 10^{-4}$.}
    \label{fig_reconstruction}
\end{figure}

Figure \ref{fig_error_regularization} shows that for small values of $h$, a high regularization parameter is a much better fit to lower the oscillations. Yet, while the reconstruction process is able to retrieve the general shape of the Kelvin transform, it fails to recover the function in a meaningful way. We use a generalized minimal residual method (GMRES) in order to solve the Tikhonov optimization problem. It proves to be more efficient than a simple conjugate gradient method in damping the oscillations. As in the work of Li \cite{CalderonNumerical} we use the $L^2-$norm as a residual norm for the numerical resolution of the Tikhonov minimization problem. An interesting perspective that might enhance the reconstruction would be to implement the Tikhonov regularization with the actual $H^{-s}$-norm. This might be challenging since the $H^{-s}$-norm involves nonlocal terms and is defined through duality.

\clearpage

\section*{Acknowledgement}
The authors would like to thank  Gis\`ele Mophou for the useful discussions and the great suggestion to include some numerical experiments in the paper.

            \bibliographystyle{plain}
              \bibliography{refs}
            
		\end{document}